\documentclass[11pt, oneside]{amsart}
\usepackage{amsfonts, amstext, amsmath, amsthm, amscd, amssymb}
\usepackage{manfnt}
\usepackage{mathrsfs}
\usepackage{tikz}
\usepackage{url}
\usepackage[all]{xypic}

\usepackage[paper=a4paper, text={138mm,208mm},centering]{geometry}

\usepackage{enumerate}
\usepackage{graphics, pinlabel, color}
\usepackage{comment}

\usepackage[all,graph]{xy}
\usepackage{hyperref}
\usepackage{psfrag}

\usepackage{float, placeins,  longtable}

\renewcommand{\setminus}{{\smallsetminus}}

\newcommand{\bp}{\begin{pmatrix}}
\newcommand{\ep}{\end{pmatrix}}
\newcommand{\be}{\begin{equation}}
\newcommand{\ee}{\end{equation}}
\newcommand{\ol}[1]{\overline{#1}}

\newcommand{\smfrac}[2]{\mbox{\footnotesize$\displaystyle\frac{#1}{#2}$}} 

\numberwithin{equation}{section}

\theoremstyle{plain}
\newtheorem{theorem}[equation]{Theorem}
\newtheorem{lemma}[equation]{Lemma}
\newtheorem{proposition}[equation]{Proposition}

\newtheorem{corollary}[equation]{Corollary}

\newtheorem*{claim*}{Claim}

\theoremstyle{definition}

\newtheorem{definition}[equation]{Definition}

\newtheorem{ex}[equation]{Example}

\numberwithin{equation}{section}

 \newtheoremstyle{TheoremNum}
        {}{}              
        {\itshape}                      
        {}                              
        {\bfseries}                     
        {.}                             
        { }                             
        {\thmname{#1}\thmnote{ \bfseries #3}}
\theoremstyle{TheoremNum}

\def\Z{\mathbb Z}

\def\Q{\mathbb Q}

\def\wt#1{\widetilde{#1}}

\def\sm{\setminus}
\def\S{\Sigma}
\def\a{\alpha}

\def\toiso{\xrightarrow{\simeq}}

\def\ll{\langle}

\def\rr{\rangle}

\def\bp{\begin{pmatrix}}
\def\ep{\end{pmatrix}}
\def\ba{\begin{array}}
\def\ea{\end{array}}
\def\bn{\begin{enumerate}}
\def\en{\end{enumerate}}

\DeclareMathOperator\cl{cl}
\DeclareMathOperator\lk{lk}

\DeclareMathOperator\Hom{Hom}

\DeclareMathOperator\Id{Id}
\DeclareMathOperator\ord{ord}

\DeclareMathOperator\im{im}

\DeclareMathOperator\coker{coker}

\newcommand{\eps}{\varepsilon}

\begin{document}

\title{The four genus of a link, Levine-Tristram signatures and satellites}

\author{Mark Powell}
\address{D\'epartement de Math\'ematiques,
Universit\'e du Qu\'ebec \`a Montr\'eal, Canada}
\email{mark@cirget.ca}


\def\subjclassname{\textup{2010} Mathematics Subject Classification}
\expandafter\let\csname subjclassname@1991\endcsname=\subjclassname
\expandafter\let\csname subjclassname@2000\endcsname=\subjclassname

\begin{abstract}
We give a new proof that the Levine-Tristram signatures of a link give lower bounds for the minimal sum of the genera of a collection of oriented, locally flat, disjointly embedded surfaces that the link can bound in the 4-ball.  We call this minimal sum the 4-genus of the link.

We also extend a theorem of Cochran, Friedl and Teichner to show that the 4-genus of a link does not increase under infection by a string link, which is a generalised satellite construction, provided that certain homotopy triviality conditions hold on the axis curves, and that enough Milnor's $\ol{\mu}$-invariants of the infection string link vanish.

We construct knots for which the combination of the two results determines the 4-genus.

\end{abstract}
\maketitle

\section{Introduction}

All links, surfaces and manifolds are oriented, all embeddings are topologically locally flat.
Let $L = L_1 \sqcup \dots \sqcup L_m$ be an oriented, ordered, $m$-component link in $S^3$.
The \emph{4-genus} of $L$ is $$g_4(L) = \min\Big\{\sum_{i=1}^m g_i \,\Big| \, g_i = g(\S_i),\, \S_1\sqcup \dots \sqcup \S_m \hookrightarrow D^4,\,  \partial \S_i=L_i  \Big\}.$$
The minimum is taken over topological locally flat embeddings of a disjoint collection of oriented surfaces into the 4-ball $D^4$, with oriented boundary $L$.
We have $g_4(L)=0$ if and only if $L$ is slice, and we write $g_4(L)=\infty$ for the minimum of the empty set.



In Theorem~\ref{theorem:LT-sig-4-ball-genus} we prove that the Levine-Tristram signatures give a lower bound for the 4-genus of a link, and we show in Theorem~\ref{Thm:our_main} that certain infection operators do not increase the 4-genus.  The two techniques yield upper and lower bounds for the 4-genus.  We combine the results to give new examples of knots for which our  techniques are able to determine the 4-genus.

\subsection{Levine-Tristram signatures and four genus}

Let $F$ be a connected Seifert surface for $L$ in $S^3$, and let $V \colon H_1(F;\Z) \times H_1(F;\Z) \to \Z$ be the Seifert form; let us also use $V$ to denote a matrix representative for the form in terms of a basis for $H_1(F;\Z)$.  The matrix
$$B(t) := (1-t)V + (1-t^{-1})V^T$$ determines a sesquilinear, hermitian form over $\Q(t)$, the field of fractions of the Laurent polynomial ring $\Z[\Z]=\Z[t,t^{-1}]$ (here we consider $\Z[t,t^{-1}]$ and $\Q(t)$ to be rings with involution by extending $\ol{t} = t^{-1}$).
For any complex number $z \in S^1$, $B(z)$ is a hermitian matrix over $\mathbb{C}$, with respect to the usual complex conjugation involution.  We may consider its signature $\sigma(B(z))$.

\begin{definition}\label{defn:LT-signature}
  The \emph{Levine-Tristram signature} of $L$ at $z= e^{i\theta} \in S^1$ is defined to be the average of the one-sided limits:
  \[\sigma_L(z) := \smfrac{1}{2}\Big(\lim_{\omega \to \theta_+} \sigma(B(e^{i\omega})) + \lim_{\omega \to \theta_-} \sigma(B(e^{i\omega}))\Big).\]
\end{definition}

The Levine-Tristram signature at $z$ turns out to be independent of the choice of Seifert surface $F$ and matrix representative~$V$~\cite{Tristram:1969-1}.   Note that the matrix $B(z)$ can have some zero eigenvalues, for example when $z=1$ and also whenever $\Delta_L(z)=0$. Here $\Delta_L(t):=\det(tV-V^T)$ is the Alexander polynomial of~$L$ (but see the official definition below).   However the signature is still defined at these values.  Taking the average of the one-sided limits as in Definition~\ref{defn:LT-signature} arranges that $\sigma_L(z)$ depends only on the concordance class of $L$.  In fact, for any $z \in S^1$ that is a root of some polynomial $p(t) \in \Z[t,t^{-1}]$ with $p(t) = p(t^{-1})$ and $|p(1)|=1$, there exists a slice knot $J$ with Alexander polynomial $\Delta_J(t) =p(t)$ whose signature function, without averaging, is nonzero at $z$~\cite{Cha-Liv-04}.  

The  Levine-Tristram signatures define a homomorphism $\mathcal{C} \to \mathbb{Z}^{\infty}$, where $\mathcal{C}$ is the knot concordance group~\cite{Levine:1969-1}, \cite{Levine:1969-2}.
For links, we remark that changing all the orientations of the components fixes the Levine-Tristram signatures.  However changing the orientation on a proper subset of the components can change the signatures in a much less predictable way.

To state the theorem relating Levine-Tristram signatures to the 4-genus we need a couple more definitions.
For an oriented, ordered $m$-component link $L \subset S^3$, let $X_L := S^3 \sm \nu L$ be the link exterior.  Define a homomorphism $\pi_1(X_L) \to \Z^m \to \Z$ by the abelianisation composed with the map sending $(x_1,\dots,x_m) \mapsto \sum_{i=1}^m x_i$.  The resulting twisted homology $H_1(X_L;\Z[\Z])$ is called the \emph{Alexander module} of $L$.

\begin{definition}\label{defn:nullity-link}
  The \emph{nullity} of a link $L$ is defined to be the rank of its Alexander module:
  \[\beta(L) := \dim(H_1(X_L;\Q(t))).\]
\end{definition}

We remark that $0 \leq \beta(L) \leq m-1$.  See \cite[Lemma~4.1]{BFP-14} for the argument, which is well-known to the experts.
We also remark that this definition of the nullity differs by one from the analogous definition used in, for example, \cite{Kauffman-Taylor:1976-1} and \cite{Murasugi:1965-1}.  For an application of $\beta(L)$ to lower bounds on the genera of cobordisms between links, see~\cite{Friedl-Powell-2014}.

\begin{definition}\label{defn:alex-poly-link}
  The \emph{Alexander polynomial} of a link $L$ is defined to be the order of its Alexander module:
  \[\Delta_L(t) := \ord_{\Z[\Z]}(H_1(X_L;\Z[\Z])).\]
\end{definition}


We will prove the following lower bound for the 4-genus of a link.

\begin{theorem}\label{theorem:LT-sig-4-ball-genus}
For any $z \in S^1$, $$|\sigma_L(z)| +m-1-\beta(L) \leq 2g_4(L).$$
\end{theorem}

Note that $\Delta_L(t) = 0$ if and only if $\beta(L) \neq 0$.
We obtain the following corollary.

\begin{corollary}\label{cor:delta-nonzero-lower-bound}
  Suppose that $\Delta_L(t) \neq 0$.  Then for any $z \in S^1$,
 $$|\sigma_L(z)| +m-1\leq 2g_4(L).$$
\end{corollary}


Here is a discussion of related prior results.
Murasugi~\cite[Theorem~9.1]{Murasugi:1965-1} showed in the smooth case\footnote{As pointed out by Pat Gilmer, the signature can jump at $-1$ for links~\cite{Gilmer-Livingston-2015}.  Murasugi does not use the averaged signature, so his bound $|\sigma(B(-1))|$ can be stronger than $|\sigma_L(-1)|$.  The signature at $-1$ is a link concordance invariant, without averaging, which may not agree with the averaged invariant.} that $|\sigma(B(-1))| \leq 2g_4(L)$.
For $z=-1$, our result is strictly stronger than this unless $\beta(L)$ is the maximal $m-1$.  Murasugi's proof involved counting Morse critical points and does not work for topological locally flat embeddings.
For $z=-1$, Theorem~\ref{theorem:LT-sig-4-ball-genus} is similar to~\cite[Corollary~3.11]{Kauffman-Taylor:1976-1}.
Kauffman and Taylor in~\cite[Corollary~3.11]{Kauffman-Taylor:1976-1} had an extra assumption that the Conway polynomial $\nabla_L(-1) \neq 0$, which in Corollary~\ref{cor:delta-nonzero-lower-bound} we are able to replace with the weaker assumption that the Alexander polynomial is nonzero, since we look at the averaged signature.   Kauffman and Taylor did consider the extension to topological locally flat embeddings, but as mentioned above only for $z=-1$.  In the case of connected surfaces in the $4$-ball, Kauffman extended the results of \cite{Kauffman-Taylor:1976-1} to $z$ any prime power root of unity in \cite[Theorem~4.1]{Kauffman:1978}. For knots and all $z$, the result essentially follows from Taylor~\cite{Taylor-genera-79}, which also deals with high dimensional knots. See also Livingston~\cite[Appendix~A]{Liv-4-genus}.  However Taylor and Livingston state their results for $PL$ locally flat or smooth embeddings, and topological locally flat embeddings are not explicitly dealt with.  A similar result to ours, for links and for all $z$, but in the smooth case only, can be found in Cimasoni-Florens \cite[Theorem~7.2]{Cimasoni-Florens-2008}.
For topologically locally flat embeddings, for all $z$, and in the link case, the result seems, technically speaking, to be new.

The previous paragraph notwithstanding, even if one were only interested in the smooth category, the proof presented here uses less machinery and is substantially different from the previously known proofs of the related results discussed.  We use regular covers instead of branched covers, and bordism theory replaces the $G$-signature theorem.  The author therefore felt that this proof should be recorded.  Perhaps the present version will be more conducive to generalisation.

In \cite{Liv-4-genus}, a stronger lower bound on the 4-genus of a knot, derived from the Levine-Tristram signature function, was defined.
Lower bounds on the 4-genus derived from higher order $L^{(2)}$-signature invariants were introduced and studied in~\cite{Cha:2006-1}.


\subsection{Infection by a string link and four genus}

Fix $m$ points, $p_1,\dots,p_m \in D^2$.  An $m$-component \emph{string link} $L$ is an embedding $L \colon \{p_1,\dots,p_m\} \times I \hookrightarrow D^2 \times I$ such that $(p_i,j) \mapsto (p_i,j)$ for $i=1,\dots m$ and $j=0,1$.

An $r$-multi-disc $\mathbb{E}$ is the standardly oriented disc $D^2$ containing $r$ ordered embedded open discs $D_1,\dots,D_r$.  To perform a infection by a string link, we need an embedding of an $r$-multi-disc $\phi \colon \mathbb{E} \to S^3$ which intersects a link $L$ only in the interiors of the $D_i$.  Denote the image of $\phi$ by $\mathbb{E}_\phi$.   Remove a thickened copy this, $\mathbb{E}_\phi \times I$, from $S^3$ and replace it by a $D^2 \times I$ which contains a string link $J$.  We call the resulting link \emph{infection of $L$ by the string link $J$, along $\mathbb{E}_\phi$,} and denote it by $S(L,J,\mathbb{E}_\phi)$, or sometimes just $S(L,J)$.  The components of the unlink in $S^3$ defined by the closed curves $\eta_1:= \partial \ol{D_1},\dots, \eta_r := \partial \ol{D_r}$ are called the axes.

The following theorem was proved in \cite{CFT09}, making crucial use of results from \cite{FT95II}.

\begin{theorem}[Cochran-Friedl-Teichner]\label{Thm:CFTmain}
Let $D = D_1 \sqcup \dots \sqcup D_m \hookrightarrow D^4$ be slice discs for a link $L$ in $S^3$.  Let $\phi \colon \mathbb{E} \to S^3$ be a map of an $r$-multi-disc such that $\eta_1,\dots,\eta_r$ bound a set of immersed discs $\delta_1,\dots,\delta_r$ in $D^4 \setminus D$ in general position.  Let $c$ be the total number of intersection and self-intersection points amongst the $\delta_i$, and let $J$ be an $r$-component string link whose closure $\widehat{J}$ has vanishing Milnor's $\ol{\mu}$-invariants up to and including length $2c$, $\ol{\mu}_{\widehat{J}}(I) = 0$ for $|I| \leq 2c$.  Then the infection link $S(L,J,\mathbb{E}_\phi)$ of $L$ by the string $J$ along $\mathbb{E}_\phi$ is also slice.
\end{theorem}

We extend Theorem \ref{Thm:CFTmain} to give the corresponding result for the 4-genus of a link, showing, under some homotopy triviality assumptions, that the 4-genus does not increased under the operation of infection by a string link.

\begin{theorem}\label{Thm:our_main}
Let $\S = \S_1 \sqcup \dots \sqcup \S_m \hookrightarrow D^4$ be oriented embedded surfaces in $D^4$ whose boundary is an ordered link $L$ in $S^3$, with genera $g_1,\dots,g_m$ respectively.  Let $\phi \colon \mathbb{E} \to S^3$ be a map of an $r$-multi--disc such that $\eta_1,\dots,\eta_r$ bound a set of immersed discs $\delta_1,\dots,\delta_r$ in $D^4 \setminus \S$ in general position.  Let $c$ be the total number of intersection and self-intersection points amongst the $\delta_i$, and let $J$ be an $r$-component string link whose closure $\widehat{J}$ has vanishing Milnor's $\ol{\mu}$-invariants up to and including length $2c$, $\ol{\mu}_{\widehat{J}}(I) = 0$ for $|I| \leq 2c$.  Then the infection link $S(L,J,\mathbb{E}_\phi)$ of $L$ by the string link $J$ along $\mathbb{E}_\phi$ is also the boundary of a collection of oriented embedded surfaces in $D^4$ with the same genera.
\end{theorem}

This has the following corollary when restricted to the case of a single satellite.

\begin{corollary}
Let $\S = \S_1 \sqcup \dots \sqcup \S_m \hookrightarrow D^4$ be embedded oriented surfaces in $D^4$ whose boundary is an ordered link $L$ in $S^3$, with genera $g_1,\dots,g_m$ respectively.   Let $\eta$ be a closed curve in $S^3 \setminus L$, which is unknotted in $S^3$ and such that $\eta$ is trivial in $\pi_1(D^4 \setminus \S)$.  Then the satellite link $S(L,J,\eta)$ of $L$ with companion $J$ and axis $\eta$ is also the boundary of a collection of embedded oriented surfaces in $D^4$ with the same genera, for any knot $J$.
\end{corollary}

\subsection{Organisation of the paper}

In Section \ref{section:example} we construct examples of knots for which the upper and lower bounds provided by the two theorems above enable us to compute the 4-genus, and for which other methods for creating an upper bound do not seem to work.  In particular we compare Theorem~\ref{Thm:our_main} with recent results on the topological 4-genus of~\cite{Feller-2015},~\cite{Feller-McCoy:2015} and~\cite{BFLL:2015}.

The proof of Theorem~\ref{theorem:LT-sig-4-ball-genus} is given in Sections \ref{section:bordism-and-witt-groups}, \ref{section:construction-of-4-manifolds} and \ref{section:proof-of-LT-sign-thm}.
Then the proof of Theorem~\ref{Thm:our_main} is given in Sections~\ref{section:characterisation} and \ref{section:proof-of-infection-by-string-link-theorem}.

\subsection*{Dedication}

This paper was written for a special memorial edition of the Journal of Knot Theory and its Ramifications, in honour of Tim Cochran.  Like much of my mathematical output, many of the main ideas in the paper have roots in Tim's papers and the ideas contained in them.  To me, Tim was a great mathematician, one of my r\^{o}le models, he was passionate about sharing his vision, and above all he was a generous and thoughtful person.

\section{Combining the upper and lower bounds to compute new 4-ball genera of knots}\label{section:example}

As an application of the above two theorems, we construct knots whose 4-genus we are now able to compute.
Start with the knot $K$ shown in Figure~\ref{figure:4-genus-example-1}, whose 4-genus is four.
The knot $K$ is constructed from a connect sum of the torus knot $T_{3,5}$ and the ribbon knot $11n139$.

\begin{figure}[h]
\begin{center}
\begin{tikzpicture}
\node[anchor=south west,inner sep=0] at (0,0){\includegraphics[scale=0.12]{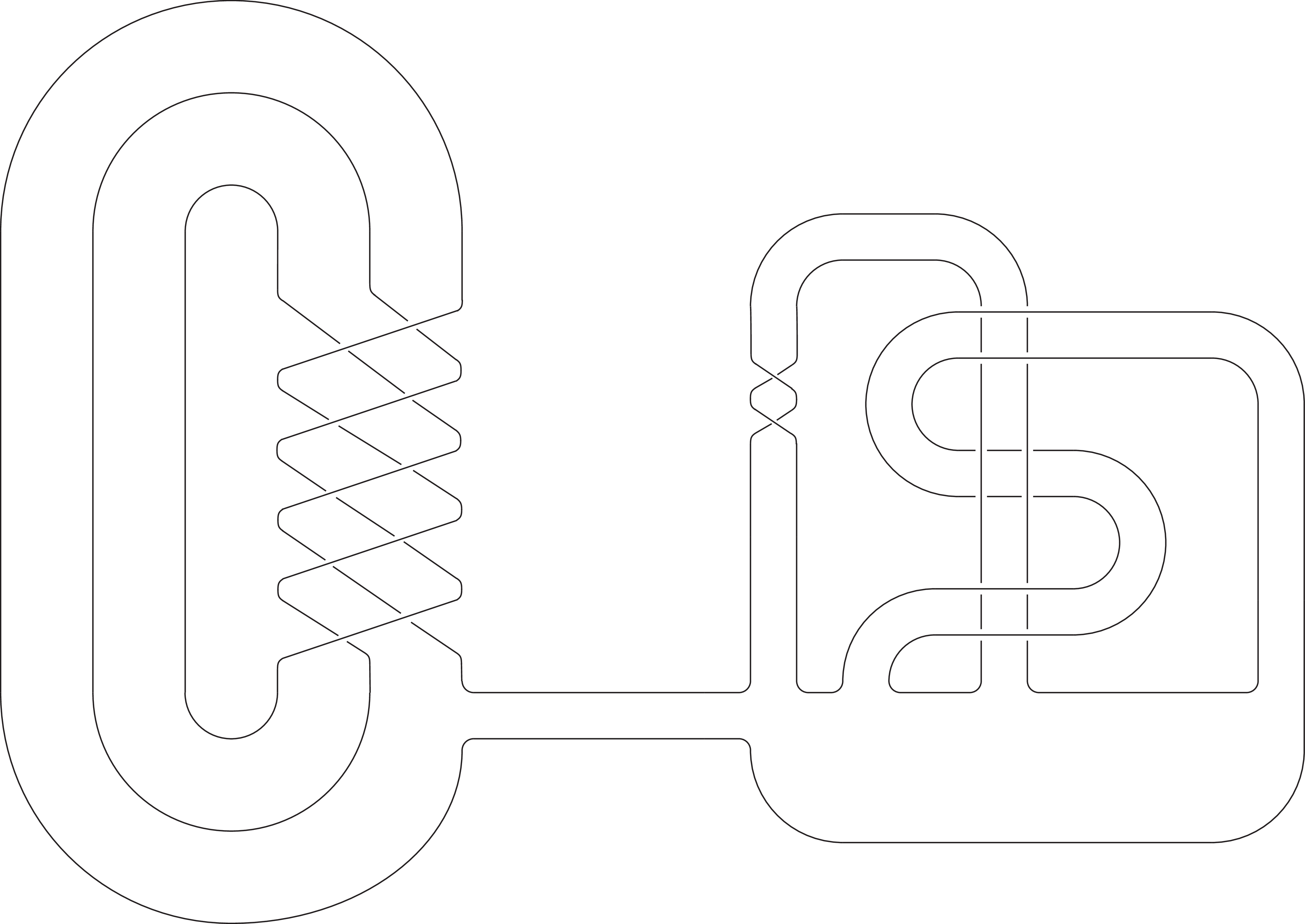}};
\end{tikzpicture}
\end{center}
    \caption{The connect sum $K=T(3,5) \# 11n139$.}
    \label{figure:4-genus-example-1}
\end{figure}

Then consider the two curves $(\eta_1,\eta_2)$ shown in Figure~\ref{figure:4-genus-example-3}.  In the figure, each box indicates a number of full right-handed twists in each of the parallel strands that pass through the box.  Thus we have infection data for each choice of integers $k$,$\ell$, $n$ and $m$.   Note that $(\eta_1,\eta_2)$ is a 2-component unlink.

\begin{figure}[h]
\begin{center}
\begin{tikzpicture}
\node[anchor=south west,inner sep=0] at (0,0){\includegraphics[scale=0.18]{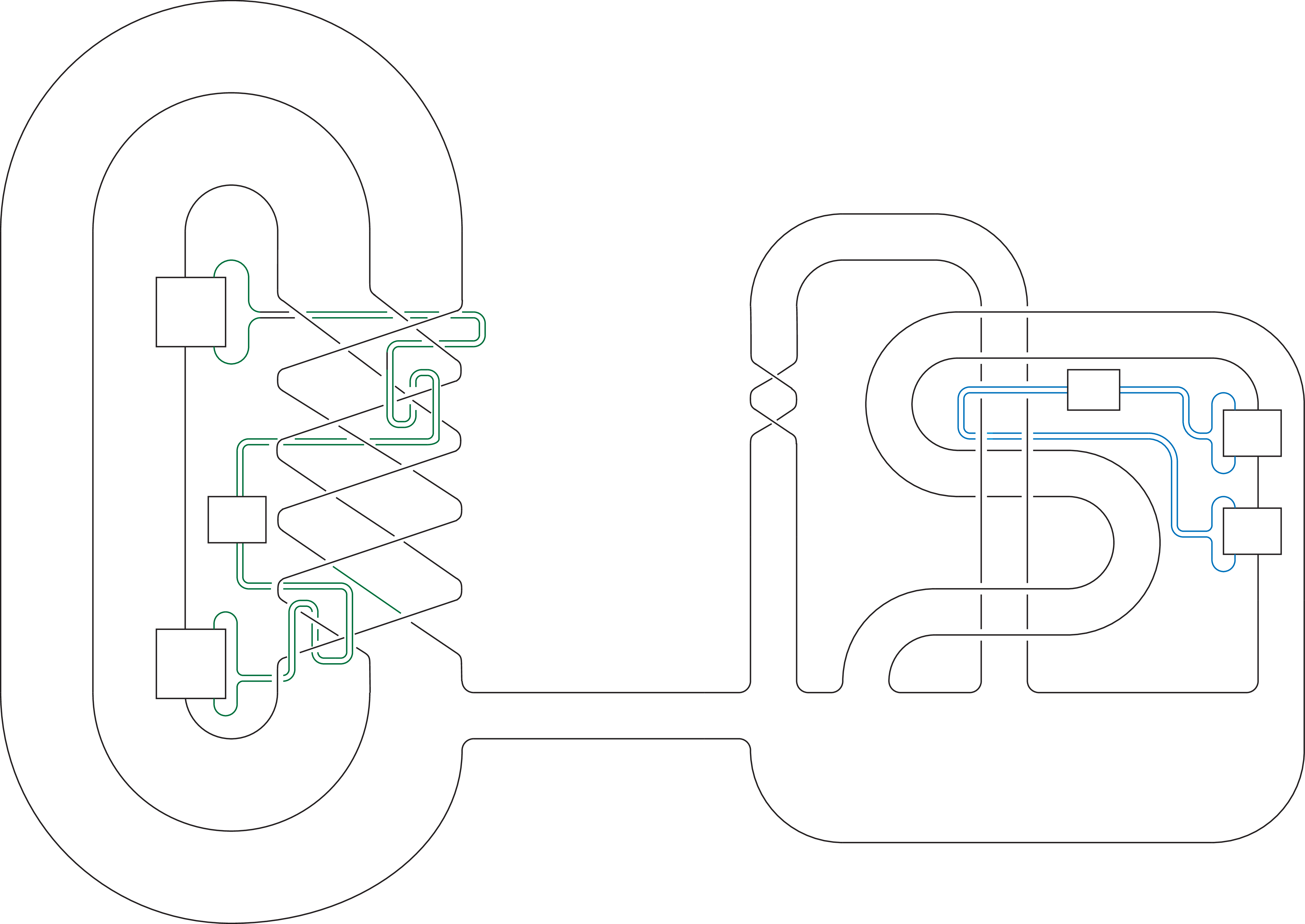}};
\node at (1.9,6.07)  {$k$};
\node at (1.87,2.57)  {$-k$};
\node at (12.42,4.87)  {$\ell$};
\node at (12.37,3.87)  {$-\ell$};
\node at (10.85,5.3)  {$n$};
\node at (2.38,4)  {$m$};
\node at (5.15,6)  {$\eta_1$};
\node at (12,3.25)  {$\eta_2$};

\end{tikzpicture}
\end{center}
    \caption{The knot $K$ together with the infection curves $\eta_1$ and $\eta_2$.}
    \label{figure:4-genus-example-3}
\end{figure}
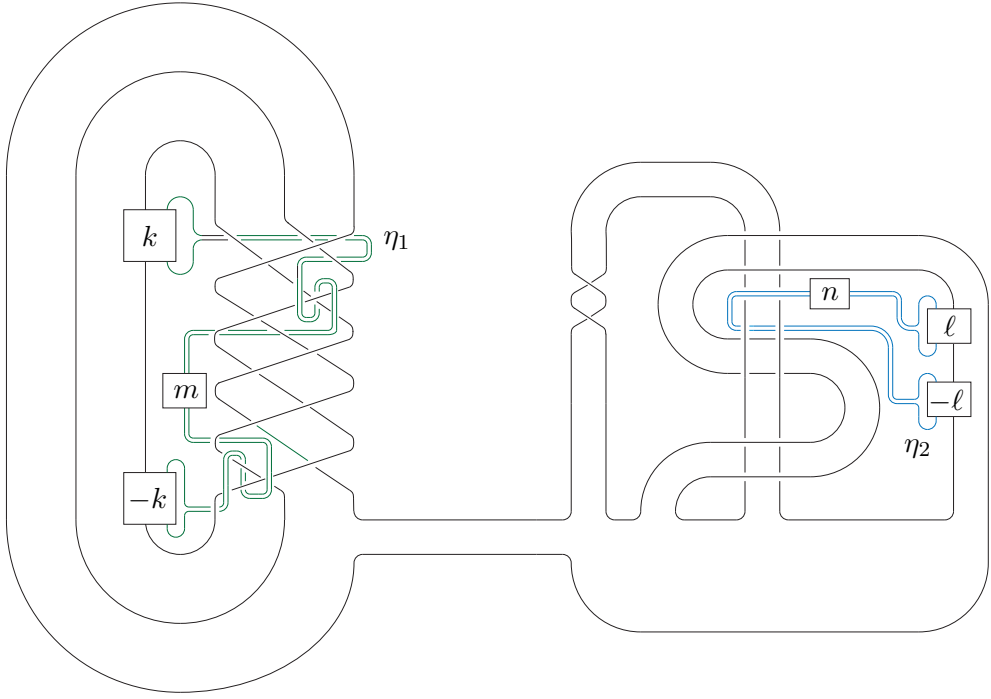

Perform infection on $K$ by a string link $J$ using the curves $(\eta_1,\eta_2)$ as data for the infection.  To define an embedding of a multi-disc, one also needs to choose a path between the $\eta$ curves that misses the disjointly embedded discs they bound and $K$.  The resulting satellite knot also depends on this choice.  Choose $J$ so that $\ol{\mu}_J(I) = 0$ for all multi-indices $I$ with  $|I| \leq 2(n+m)$. We could use one of Milnor's links from \cite[Figure~1,~p.~301]{Milnor:1957-1} for $J$.  Denote the resulting link by $S(K,J)$.  We omit the multi-disc from the notation.

\begin{proposition}\label{prop:4-genus-S(K,J)-equal-two}
  The knot $S(K,J)$ has 4-genus equal to four.
\end{proposition}

\begin{proof}
  The original seed knot $K$ has Levine-Tristram signature $|\sigma_K(-1)| = 8$, arising from the $T_{3,5}$ summand, so the 4-genus of $K$ is at least four by Theorem~\ref{theorem:LT-sig-4-ball-genus}.  Indeed the knot $K$ has 4-genus four, since eleven band moves, as indicated in Figure~\ref{figure:4-genus-example-2}, produce a $4$-component unlink.

\begin{figure}[h]
\begin{center}
\begin{tikzpicture}
\node[anchor=south west,inner sep=0] at (0,0){\includegraphics[scale=0.15]{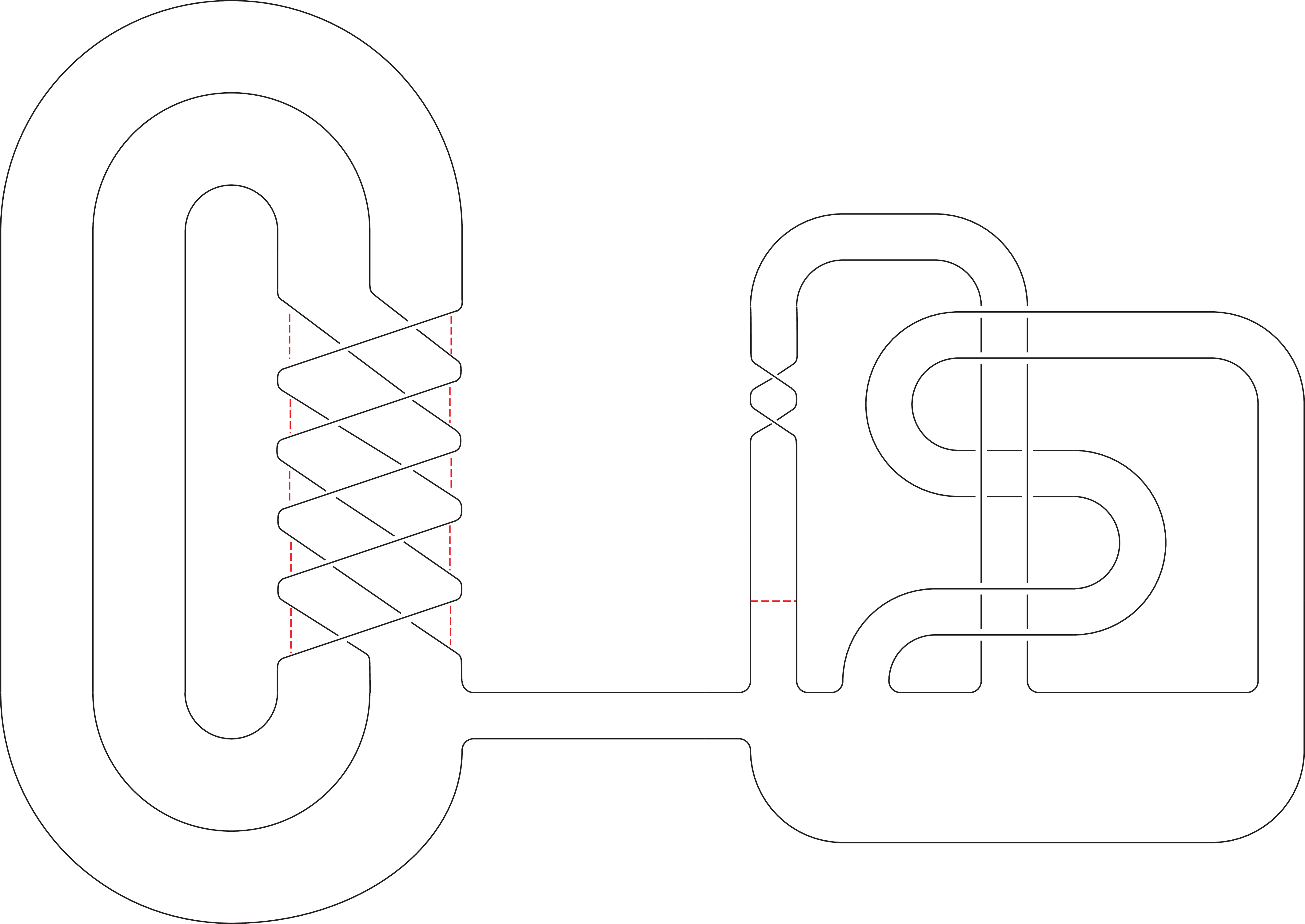}};
\end{tikzpicture}
\end{center}
    \caption{The connect sum $K= T(3,5) \# 11n139$ with places to perform eleven band moves indicated, after which one is left with a 4-component unlink.}
    \label{figure:4-genus-example-2}
\end{figure}

  To see this, observe that the band moves give rise to a smoothly embedded surface $\Sigma$ in the 4-ball with boundary $K$ and euler characteristic $4-11=-7$.  Thus the genus $g(\Sigma)$ satisfies $1-2g(\S)=-7$ and so $g(\Sigma)=4$.  The curves $(\eta_1,\eta_2)$ lie in the the commutator subgroup $\pi_1(X_K)^{(1)}$, where $X_K:=S^3 \sm \nu K$ is the knot exterior.  This implies that there is a Seifert surface whose Seifert form is not altered by the infection operation. Therefore the Alexander polynomial and the Levine-Tristram signatures are unchanged, so $g_4(S(K,J)) \geq 4$ by Theorem~\ref{theorem:LT-sig-4-ball-genus}.  Finally, the string link construction satisfies the hypotheses of Theorem~\ref{Thm:our_main}: $\eta_1$ and $\eta_2$ can be seen to be null-homotopic in the complement of the $4$-component unlink that arises after the band moves of Figure~\ref{figure:4-genus-example-2} have been performed.  The null-homotopies have $n+m$ double points, arising from undoing the twists in the boxes labelled $n$ and $m$, and we chose $J$ to have all $\ol{\mu}_J(I) = 0$ for $|I| \leq 2(n+m)$.  Thus by Theorem~\ref{Thm:our_main} the infection operation does not increase the 4-genus, so we have $g_4(S(K,J)) \leq 4$.  This completes the proof of Proposition~\ref{prop:4-genus-S(K,J)-equal-two}.
\end{proof}

We briefly discuss some other potential approaches to computing the 4-genus of $S(K,J)$.

\begin{enumerate}
  \item   The Alexander polynomial is $$\Delta_K = \Delta_{T(3,5)} \cdot \Delta_{11n139} = (t^8-t^7+t^5-t^4+t^3-t+1)(2t^2-5t+2).$$ Due to the contribution from $\Delta_{T(3,5)}$, this not satisfy the Fox-Milnor condition of factorising in the form $\pm t^kf(t)f(t^{-1})$.  Thus $S(K,J)$ is not slice and the 4-genus of $S(K,J)$ is at least one.
 \item    To show that the 4-genus at least four, one needs Theorem~\ref{theorem:LT-sig-4-ball-genus}.  Of course, since we use the Levine-Tristram signature at $-1$, this already follows from~\cite{Kauffman-Taylor:1976-1}.
  \item Knot concordance invariants associated to Khovanov and Heegaard Floer homology and gauge theory fail to give any information for locally flat embeddings.
  \item Feller~\cite{Feller-2015} showed that the width of the Alexander polynomial is an upper bound for twice the 4-genus of a knot.  The Alexander polynomial of $S(K,J)$ has width $10$, so only gives an upper bound of $5$ for the 4-genus.  
\item  Provided $k,\ell \neq 0$ and $J$ is suitably non-trivial, for example if $J$ has a nonvanishing Milnor's invariant, the knot $S(K,J)$ does not seem to be obviously concordant to a knot with Alexander polynomial of width 8.  However, one can find a genus $4$ smooth cobordism in $S^3 \times I$ to a link that can be sliced using Theorem~\ref{Thm:CFTmain}.  So to use the full extra power of Theorem~\ref{Thm:our_main} we would need a more sophisticated example.
\item  The infection curves $(\eta_1,\eta_2)$ intersect the obvious minimal genus Seifert surface for $K$, therefore only by stabilising this Seifert surface can we easily understand a Seifert surface for $S(K,J)$.  Lukas Lewark informs me that one can  apply the techniques of \cite{BFLL:2015} to reduce this stabilised Seifert surface, pushed into the 4-ball, to a genus 4 surface.  It would be interesting to find some examples where the approach of \cite{Feller-2015}, \cite{BFLL:2015} of excising a part of the Seifert surface with Alexander polynomial one, fails to find a sharp upper bound, but Theorem~\ref{Thm:our_main} does.
\end{enumerate}

\section{Bordism groups and Witt groups}\label{section:bordism-and-witt-groups}

In this section we begin the proof of Theorem~\ref{theorem:LT-sig-4-ball-genus}, which will take until the end of Section~\ref{section:proof-of-LT-sign-thm}.

The Atiyah-Hirzebruch spectral sequence~\cite{AHSS-paper}, together with the facts that $\Omega^{STOP}_n= 0$ for $n=1,2,3$, implies that $\Omega_3^{STOP}(B\Z) =0$ and $\Omega_4^{STOP}(B\Z) \cong \Omega_4^{STOP}$.  We then have $\Omega_4^{STOP} \cong \Z \oplus \Z_2$, with the summands detected by the signature and the Kirby-Siebenmann invariant $ks$ respectively; this is described in~\cite[Section~5.1]{teichnerthesis}.  The Atiyah-Hirzebruch sequence uses topological transversality~\cite[Theorem~9.5]{Freedman-Quinn:1990-1}.

Let $(M,\phi)$ be a connected closed oriented 3-manifold with a map $\phi \colon M \to B\Z$.  Then since $\Omega_3^{STOP}(B\Z) =0$, $M$ is the boundary of a connected topological $4$-manifold $W$ with a map $\Phi \colon W \to B\Z$ extending $\phi$.   We may consider the twisted homology groups $H_*(M;\Q(t))$ and $H_*(W;\Q(t))$, corresponding to the maps $\pi_1(M) \to \Z$ and $\pi_1(W) \to \Z$ induced by $\phi$ and $\Phi$ respectively.  We also will consider the intersection form $\lambda_W \colon H_2(W;\Q(t)) \times H_2(W;\Q(t)) \to \Q(t)$, which is a hermitian sesquilinear form.

Since $H_i(M;\Q(t))$ may be nonzero, in particular for $i=1,2$, the form $\lambda_W$ may be singular.
To obtain a nonsingular form, instead of $\lambda_W$ we consider the restricted intersection form.  Define
\[H_2(W)^{\dag} := H_2(W;\Q(t))/\im(H_2(M;\Q(t))).\]
The intersection form of $W$, $\lambda_W \colon H_2(W;\Q(t)) \times H_2(W;\Q(t)) \to \Q(t)$,
induces a hermitian, sesquilinear form
\[\lambda^{\dag}_W \colon H_2(W)^\dag \times H_2(W)^\dag \to \Q(t).\]
This is well-defined by the next lemma.
We will use this restricted intersection form to define signature invariants of $(M,\phi)$.


\begin{lemma}\label{lemma:lambda-dag-nonsingular}
  The intersection form $\lambda^\dag \colon H_2(W)^\dag \times H_2(W)^\dag \to \Q(t)$ is well-defined and nonsingular.
\end{lemma}

\begin{proof}
  Consider the following commuting diagram. In the diagram and in the rest of the proof, all homology and cohomology groups are with $\Q(t)$ coefficients.  For a $\Q(t)$-module $P$, let $P^*:= \Hom_{\Q(t)}(P;\Q(t))$.
  \[\xymatrix @C-0.5cm{
  H_3(W,M) \ar[r] \ar[d]^-{\cong} & H_2(M) \ar[r]^-{i_*} \ar[d]^-{\cong} & H_2(W) \ar[r] \ar[dr]^{\wt{\lambda}} \ar[d]^-{\cong} & H_2(W,M) \ar[r] \ar[d]^-{\cong} & H_1(M)  \ar[d]^-{\cong} 
  \\
  H^1(W) \ar[r]  & H^1(M) \ar[r] & H^2(W,M) \ar[r]  & H^2(W) \ar[r] \ar[d]_-{\kappa}^-{\cong}  & H^2(M) \ar[d]_-{\kappa}^-{\cong}
  \\
   & & & H_2(W)^* \ar[r]^{i^*} & H_2(M)^* 
   }\]
  The vertical maps are isomorphisms.  The top vertical maps are given by Poincar\'{e} duality, and the lower vertical maps are the Kronecker evaluation maps $\kappa$ from the universal coefficient theorem. Define the map $\wt{\lambda}$ via the diagram. The adjoint of the intersection pairing, which we also denote by $\lambda_W$, is given by
  \[\lambda_W = \kappa \circ \wt{\lambda} \colon H_2(W) \to H_2(W)^*.\]
It follows from exactness and the fact that the vertical maps are isomorphisms that
\[\ker \lambda_W = \im(i_* \colon H_2(M) \to H_2(W)).\]
Thus $\lambda_W$ factors as $H_2(W) \to H_2(W)/\im(H_2(M)) \to H_2(W)^*$, as shown on the top row of the next diagram.
By exactness and commutativity of the diagram above, we see that $i^* \circ \lambda_W = i^* \circ \kappa \circ \wt{\lambda} = 0$.  Thus for any $v \in H_2(W)$, $\lambda_W(v)(w)=0$ whenever $w \in \im (i_*)$.  This implies that $\lambda^{\dag}_W$ is well-defined.  Thus we have a commutative diagram as follows:
\[\xymatrix{H_2(W) \ar[r] & \frac{H_2(W)}{\im(H_2(M))} \ar@{^{(}->}[r] \ar[dr]_-{\lambda^{\dag}_W} & H_2(W)^* \\ & & \Big(\frac{H_2(W)}{\im(H_2(M))}\Big)^*. \ar[u] }\]
Then since the top right horizontal map is injective, so is $\lambda^{\dag}_W$.  The domain and codomain of $\lambda^\dag_W$ are vector spaces over $\Q(t)$ of the same dimension, and so $\lambda^\dag_W$ is an isomorphism.
\end{proof}

 By Lemma~\ref{lemma:lambda-dag-nonsingular}, $(H_2(W)^\dag,\lambda^{\dag}_W)$ determines an element of the Witt group $L^0(\Q(t))$ of nonsingular hermitian sesquilinear forms.  By definition, two forms are equivalent in the Witt group if they become isometric after stablising one or both with a finite number of copies of the hyperbolic form
 $$\left(\Q(t)^2, \begin{pmatrix}
  0 & 1 \\ 1 & 0
\end{pmatrix}\right).$$  Addition of forms is by direct sum and the inverse of a form $(\Q(t)^n,\lambda)$ is $(\Q(t)^n,-\lambda)$.

Since $\Omega_4^{{STOP}}(B\Z) \cong \Z\oplus \Z_2$, any two choices of 4-manifold $W,W'$ with the same signature and Kirby-Siebenmann invariant are cobordant relative to their boundaries over $B\Z$.

\begin{lemma}\label{lemma:witt-class-independence}
  The intersection forms $\lambda_{W}^\dag$ and $\lambda_{W'}^{\dag}$ represent the same element in $L^0(\Q(t))$.
\end{lemma}

\begin{proof}
In this proof, again all homology and cohomology groups are with $\Q(t)$~coefficients.
  Define $V:=W \cup_{M} -W'$ and let $U^5$ be a cobordism between $W$ and $W'$ rel.\ boundary, that is a null cobordism of $V$.
Define $P:= \ker(H_2(V) \to H_2(U))$.  Then by a standard argument $P$ is a Lagrangian subspace for the intersection form $\lambda_{V} \colon H_2(V) \times H_2(V) \to \Q(t)$ of $V$.  We elaborate slightly on this standard argument, for the convenience of the reader.  Consider the homology exact sequence of the pair $(U,V)$. A dimension counting argument, together with Poincar\'{e}-Lefschetz duality and universal coefficients, shows that $P$ is a half rank subspace of $H_2(V;\Q(t))$.  The commutative diagram  below, with exact rows and vertical isomorphisms deriving from Poincar\'{e}-Lefschetz duality and universal coefficients, can be used to show that the intersection form, whose adjoint is the middle vertical map, vanishes on $P$.  As in the previous proof, for a $\Q(t)$-module $N$ we denote $\Hom_{\Q(t)}(N,\Q(t))$ by $N^*$.
\[\xymatrix{H_3(U,V) \ar[r] \ar[d] & H_2(V) \ar[r] \ar[d] & H_2(U) \ar[d] \\
H_2(U)^* \ar[r] & H_2(V)^* \ar[r] & H_3(U,V)^*}\]

\begin{claim*}
  The intersection form $\lambda_V$ is Witt equivalent to $\lambda_W^{\dag} \oplus -\lambda_{W'}^\dag$.
\end{claim*}
From the claim and the fact from above that $\lambda_V$ is Witt equivalent to zero, the lemma follows.  The remainder of the proof comprises the proof of the claim.

Denote the inclusion induced maps on $\Q(t)$-coefficient homology by $i_* \colon H_*(M) \to H_*(W)$ and $i'_* \colon H_*(M) \to H_*(W')$.
It follows from the Mayer-Vietoris sequence for the decomposition $V = W \cup_M W'$ that $H_2(V)$ is isomorphic to
\begin{align*}
  &\coker\big(H_2(M) \xrightarrow{i_* \oplus i'_*} H_2(W)\oplus H_2(W')\big) \oplus \ker\big(H_1(M) \xrightarrow{i_* \oplus i'_*} H_1(W)\oplus H_1(W')\big)\\
  \cong & \frac{H_2(W)}{\im(i_*)} \oplus \frac{H_2(W')}{\im(i'_*)} \oplus \big(H_2(M)/(\ker(i_*) + \ker(i'_*))\big)   \oplus \big(\ker(i_*) \cap \ker(i'_*)\big).
  \end{align*}
Here is a justification of the above isomorphism. The last summands in each line are easily identified.  It remains to identify the first summand $\coker\big(H_2(M) \xrightarrow{i_* \oplus i'_*} H_2(W)\oplus H_2(W')\big)$ with the first three summands in the next line.  This follows from the general fact that for homomorphisms of vector spaces $i \colon A \to B$ and $i' \colon A \to B'$, the map
\[\ba{rcl}
\coker(A \xrightarrow{i \oplus i'} B \oplus B') & \to & \coker(i) \oplus \coker(i') \\
(b,b') + (i \oplus i')(A) & \mapsto & (b+ i(A), b' + i'(A))
\ea\]
is surjective with kernel $A/(\ker(i)+\ker(i'))$.

Elements of $H_2(M)/(\ker(i_*) + \ker(i'_*)) =\im(H_2(M) \to H_2(V))$ form a submodule on which the intersection form of $V$ vanishes, since given two representative surfaces, one can be pushed into $W$ slightly, to make them disjoint.  Elements of $H_2(W)/\im(i_*)$  and $H_2(W')/\im(i'_*)$ intersect trivially with elements of $\im(H_2(M) \to H_2(V))$.   The intersection form of $V$ restricted to $H_2(W)/\im(i_*) \oplus H_2(W')/\im(i'_*)$ is a direct sum  $\lambda_{W}^\dag \oplus -\lambda_{W'}^\dag$, and since each is nonsingular by Lemma~\ref{lemma:lambda-dag-nonsingular}, we can make a change of basis so that $\lambda_V$ has $\lambda_{W}^\dag \oplus -\lambda_{W'}^\dag$ as an orthogonal direct summand.
It therefore suffices to see that the form on
\[\big(H_2(M)/(\ker(i_*) + \ker(i'_*))\big)   \oplus \big(\ker(i_*) \cap \ker(i'_*)\big)\]
is Witt trivial.  For this, since the intersection form vanishes on the first summand, it suffices to see that this summand is of half rank.
We have:
\begin{align*}
 & \im\big(H_2(M) \to H_2(V)\big) \\
 \cong & H_2(M)/(\ker(i_*) + \ker(i'_*))\\
  \cong& \coker\big(H_3(W,M) \oplus H_3(W',M) \xrightarrow{\partial_* \oplus \partial'_*} H_2(M)\big) \\
 \cong& \coker\big( H^1(W) \oplus H^1(W') \xrightarrow{i^* \oplus (i')^*} H^1(M)\big) \\
 \cong &\coker\big( H_1(W)^* \oplus H_1(W')^* \xrightarrow{i^* \oplus (i')^*} H_1(M)^*\big) \\
 \cong &\ker\big( H_1(M) \xrightarrow{i_* \oplus i'_*} H_1(W) \oplus H_1(W')\big)^* \\
 \cong &\big(\ker(i_*) \cap \ker(i'_*)\big)^*
\end{align*}
Here the third isomorphism uses the commutativity of the left hand square in the large commutative diagram in the proof of Lemma~\ref{lemma:lambda-dag-nonsingular}.
This completes the proof of the lemma.
\end{proof}

We can represent an element of the Witt group $L^0(\Q(t))$ by a matrix $A(t)$, and evaluate at $z \in S^1 \subset \mathbb{C}$.  For $z$ such that $\det(A(z)) \neq 0$, this determines a nonsingular hermitian matrix over $\mathbb{C}$, and we can take its signature $\sigma(A(z))$.
Define a homomorphism $L^0(\Q(t)) \to \Z$, for $z=e^{i\theta} \in S^1 \subset \mathbb{C}$ by
\[A(t) \mapsto \smfrac{1}{2}\Big(\lim_{\omega \to \theta_+} \sigma(A(e^{i\omega})) + \lim_{\omega \to \theta_-} \sigma(A(e^{i\omega}))\Big) =: \sigma(A(z)).\]
It is not too hard to see that we have a well-defined homomorphism, as follows.  For transcendental $z \in S^1$, since $A(t)$ is nonsingular it is impossible to have $\det(A(z)) = 0$.  Therefore for each such $z$ we obtain a homomorphism $L^0(\Q(t)) \to L^0(\mathbb{C})$. Then the signature gives an isomorphism $L^0(\mathbb{C}) \toiso \Z$. The one-sided limits above can be computed using only $\omega$ that give rise to transcendental~$z$.

To define a quantity that is invariant under all choices of $W$, not just those with the same signature and $ks$, we need to quotient out the Witt group by the image of the intersection forms of closed $STOP$ 4-manifolds.  However since $\Omega_4^{STOP}(B\Z) \cong \Omega_4^{STOP}$, every closed 4-manifold with a map to $B\Z$ is bordant over $B\Z$ to
$$S^1 \times S^3 \#  \big(\#^p \mathbb{C}P^2\big) \# \big(\#^q \ol{\mathbb{C}P^2}\big) \# \big(\#^r \ast\hspace{-2pt}\mathbb{C}P^2\big)$$
for some $p,q,r$. Here $\ast\mathbb{C}P^2$ is the topological 4-manifold homotopy equivalent to $\mathbb{C}P^2$ but not homeomorphic to it, with $ks(\ast\mathbb{C}P^2) =1$, of \cite[Section~10.4]{Freedman-Quinn:1990-1}.  In particular the intersection form $\lambda_V$ over $\Q(t)$ of a closed 4-manifold $V \to B\Z$ is Witt equivalent to a form tensored up from the integers; more precisely, there exists a basis of $H_2(V;\Q(t))$, with respect to which the representative matrix contains only elements of $\Z \subset \Q(t)$.  Therefore, for a closed 4-manifold $V$, we have $\sigma(A(z)) = \sigma(A(1)) = \sigma(V)$ for all $z \in S^1$.  Thus we define
\begin{equation}\label{equation:signature-defect}
  \sigma_{M}(z):= \sigma\big(A(z)\big) - \sigma(W) \in \Z.
\end{equation}
This is an invariant of $M$ up to homeomorphisms which respect $\phi$.

\section{Construction of bounding 4-manifolds}\label{section:construction-of-4-manifolds}

Let $L$ be an $m$-component oriented, ordered link in $S^3$ and let $M_L$ be the zero-framed surgery manifold of $L$.  This admits a map $\phi \colon M_L \to B\Z$ corresponding to the map $\pi_1(M_L) \to \Z^m \to \Z$ given by the abelianisation that sends the $i$th oriented meridian to the $i$th standard basis vector $e_i$, followed by the map $\Z^m \to \Z$ sending $\sum_{i=1}^m x_i e_i \mapsto \sum_{i=1}^m x_i$.  

In this section, given an oriented connected embedded surface $\S$ in $D^4$ with boundary $L$, we construct a 4-manifold $W_\S$ with $\partial W_\S=M_L$, that admits an extension of $\phi$.  Lemma~\ref{lemma:COT2}, which deals with the case that $\S$ is a connected Seifert surface for $L$ pushed into the $4$-ball, is due to Ko~\cite[pp.~538-9]{Ko:1989-1}.
The explicit statement of the lemma for knots appears in \cite[Lemma~5.4]{Cochran-Orr-Teichner:2002-1}. We state the link version below, inviting the interested reader to check that the proof given for knots in \cite[Lemma~5.4]{Cochran-Orr-Teichner:2002-1} generalises easily to the case that the boundary of $F$ has more than one component.

\begin{lemma}\label{lemma:COT2}
  Given a connected Seifert surface $F$ for $L$, there exists a null-bordism $W_F$ for $M_L$ over $B\Z$, such that $\sigma_L(z) = \sigma(\lambda_{W_F}(z)) - \sigma(W_F)$ for all $z \in S^1$.  Thus $\sigma_{M_L}(z)=\sigma_L(z)$.
\end{lemma}

The construction of $W_F$, which we will describe below, generalises to produce a 4-manifold $W_{\S}$ for any collection $\S=\S_1 \sqcup \dots \sqcup \S_m \hookrightarrow D^4$ of disjointly embedded locally flat oriented surfaces in the 4-ball with $\partial \S=L$.  For $W_F$, computation of the intersection form shows that the intersection form coincides with the matrix used to compute the Levine-Tristram signatures; this is the main step in the proof of Lemma~\ref{lemma:COT2}.  Since the signature defect of (\ref{equation:signature-defect}) is independent of the bounding 4-manifold, we obtain the same signature defect for any 4-manifold constructed using any collection of surfaces $\S$.  To prove Theorem~\ref{theorem:LT-sig-4-ball-genus} we will investigate the relationship between the genera of the surfaces $\S$ and the euler characteristic of $W_{\S}$, which in turn is related to the signatures.

Here is the construction of $W_{\S}$.
Suppose that $L = \partial \S$, where $\S = \S_1 \sqcup \dots \sqcup \S_m \hookrightarrow D^4$, and $\S_i$ has genus~$g_i$.  Define $g:= \sum_{i=1}^m g_i$.  Note that the pairwise linking numbers of $L$ must all be zero in order for such a collection of disjointly embedded surfaces to exist.

Define $Y_{\S}:= D^4 \sm \nu \S$.
Locally flat submanifolds of 4-manifolds have normal bundles, by \cite[Section~9.3]{Freedman-Quinn:1990-1}.
Note that $\partial Y_{\S} = X_L \cup_{\partial X_L} S^1 \times \S$, where $X_L$ denotes the link exterior $S^3 \sm \nu L$. The capped off surface $\S \cup \bigsqcup^m D^2$ can be mapped by a homeomorphism to the boundary of a collection of handlebodies $G=G_1 \sqcup \dots \sqcup G_m$ ($G$ is not embedded in $D^4$).  

Choose $G$ so that $\ker(H_1(\S;\Z) \to H_1(G;\Z))$ lies in $\ker(\phi| \colon H_1(\Sigma;\Z) \to \Z)$, so that the map $\phi$ extends over $G$. Note that $\ker(\phi| \colon H_1(\Sigma;\Z) \to \Z)$ is always at least a half rank summand, so such a $G$ can always be found.  

Define $W_\S:= Y_{\S} \cup_{S^1 \times \S} S^1 \times G$.    Observe that $\partial W_{\S} = M_L$.  Moreover there is an extension of $\phi \colon M_L \to B\Z$ to a map $W \to B\Z$ that induces a coefficient system $\pi_1(W) \to \Z$, which we will exploit in Section~\ref{section:proof-of-LT-sign-thm}.

We remark that the construction of $W_F$ used in the proof of Lemma~\ref{lemma:COT2} from \cite[Lemma~5.4]{Cochran-Orr-Teichner:2002-1}, for $F$ as above a connected pushed in Seifert surface, only differs in that $G$ is a single connected handlebody instead of a disjoint union of handlebodies.

\begin{lemma}\label{lemma:homology-W-S}
The rational homology of $W_\S$ is given by
  \[H_i(W_{\S};\Q) \cong \begin{cases}
    \Q & i=0 \\
    \Q^m & i=1 \\
    \Q^{2g} & i=2 \\
    0 & \text{else.}
     \end{cases}\]
In particular the euler characteristic $\chi(W_{\S}) = 2g-m+1$.  Moreover, $\sigma(W_{\S})=0$.
\end{lemma}


\begin{proof}
Use the decomposition $D^4=Y_{\S} \cup_{S^1 \times \S } D^2 \times \S$, and the associated Mayer-Vietoris sequence, to compute
\[H_i(Y_{\S};\Q) \cong \begin{cases}
    \Q & i=0 \\
    \Q^m & i=1 \\
    \Q^{2g} & i=2 \\
    0 & \text{else.}
     \end{cases}\]
     More details can be found in the proof of Proposition~\ref{Prop:A} below.
     Note that $\chi(Y_\S)=2g-m+1$.  From the formulae $\chi(A \cup B) = \chi(A) +\chi(B) - \chi(A \cap B)$ and $\chi(S^1 \times Z) =0$ (for any finite CW-complexes $A$, $B$ and $Z$), it follows that \[\chi(W_\S) = \chi(Y_{\S}) + \chi(S^1 \times G) - \chi(S^1 \times \S) = \chi(Y_{\S}) = 2g-m+1.\]

     Next use the decomposition $W_{\S} = Y_{\S} \cup_{S^1 \times \S} S^1 \times G$ to compute the homology of $W_{\S}$.
     We may assume that a summand $\Q^g \subset H_1(\S;\Q)$ dies in $H_1(G;\Q)$.  We have
     \begin{align*}
    \to & H_2(S^1 \times \S;\Q) \xrightarrow{j_2} H_2(S^1 \times G;\Q) \oplus H_2(Y_\S;\Q) \to H_2(W_\S;\Q) \to \\
       \to &H_1(S^1 \times \S;\Q) \xrightarrow{j_1} H_1(S^1 \times G;\Q) \oplus H_1(Y_\S;\Q) \to H_1(W_\S;\Q) \xrightarrow{0}
     \end{align*}
  by a straightforward computation with zeroth homology.
  This yields
 \[\xymatrix @R-0.7cm{
    \ar[r] & \Q^{2g} \ar[r]^-{j_2} & \Q^g \oplus \Q^{2g} \ar[r] & H_2(W_\S;\Q) \ar[r] & \\
       \ar[r] & \Q^{2g} \oplus \Q^m \ar[r]^-{j_1} & (\Q^g \oplus \Q^m) \oplus \Q^m \ar[r] & H_1(W_\S;\Q) \ar[r]^-{0} &
     }\]
The map $j_1\colon H_1(S^1 \times \S;\Q) \to H_1(S^1 \times G;\Q) \oplus H_1(Y_\S;\Q)$ serves to identify the generators of $\Q^m \cong H_1(Y_\S;\Q)$ with the $S^1 \times \{pt\}$ summands of
\[H_1(S^1 \times G;\Q) \cong (H_0(S^1;\Q) \times H_1(G;\Q)) \oplus (H_1(S^1;\Q) \otimes H_0(G;\Q))   \cong \Q^g \oplus \Q^m,\]
and either kills the elements of $H_0(S^1;\Q) \otimes H_1(G;\Q) \cong \Q^g$ or identifies them with elements of $H_1(Y_\S;\Q) \cong \Q^m$.  Thus $H_1(W_\S;\Q) \cong \Q^m$ as claimed.  We observe that the kernel of $j_1$ is isomorphic to $\Q^g$ by dimension counting.

The map $j_2\colon H_2(S^1 \times \S;\Q) \to H_2(S^1 \times G;\Q) \oplus H_2(Y_\S;\Q)$ determines an isomorphism when the codomain is restricted to $H_2(Y_\S;\Q)$, which the conscientious reader will have seen in the Mayer-Vietoris computation from the beginning of the proof. Thus, taking the cokernel of~$j_2$ identifies half of a generating set of $H_2(Y_\S;\Q)\cong \Q^{2g}$ with generators of $H_2(S^1 \times G;\Q) \cong \Q^g$, and kills the other half of the generators.  We obtain a short exact sequence
\[0 \to \Q^g \to H_2(W_{\S};\Q) \to \Q^{g} \to 0,\]
so $H_2(W_\S;\Q) \cong \Q^{2g}$ as claimed.
Continuing the Mayer-Vietoris computation to the left, since $\ker j_2 = 0$, the higher homology groups are easily seen to vanish.

For the last sentence of the lemma, namely $\sigma(W_\S)=0$, inspection of the generators which can be understood from the above proof shows that the ordinary intersection form on $H_2(W_{\S};\Q)$ has a $\Q^g$ direct summand, to wit the $\Q^g$ on the left of the short exact sequence above, represented by disjointly embedded surfaces.  These are of the form $\a_{2i} \times S^1$, where $\a_1,\dots,\a_{2g}$ is a symplectic basis for $H_1(\S;\Q)$ and $\ker(H_1(\S;\Q) \to H_1(G;\Q))$ is generated by the $\a_{2i-1}$ for $i=1,\dots,g$. Such a basis can always be found.  These embedded surfaces generate a Lagrangian submodule of the intersection form on $H_2(W_{\S};\Q)$, from which it follows that $\sigma(W_{\S})=0$.
\end{proof}

\section{Proof of Theorem~\ref{theorem:LT-sig-4-ball-genus}}\label{section:proof-of-LT-sign-thm}

Suppose that an ordered oriented link $L=L_1 \sqcup \dots \sqcup L_m$ satisfies $L=\partial \Sigma$, where $\Sigma= \S_1 \sqcup \dots \sqcup \S_m$ is a properly embedded, oriented, locally flat collection of surfaces in $D^4$, with genus $g$.

Recall that $X_L := S^3 \sm \nu L$, the link exterior, and let $M_L = X_L \cup_{\partial X_L} \bigcup^m S^1 \times D^2$ be the 3-manifold obtained from zero-framed surgery on~$L$.  The representation $\pi_1(X_L) \to H_1(X_L;\Z) \toiso \Z^m \to \Z$ extends to $\pi_1(M_L)$.


\begin{definition}\label{defn:zero-surgery-nullity-link}
  Define the \emph{zero-surgery nullity} of a link $L$ to be:
  \[\beta(M_L) := \dim(H_1(M_L;\Q(t))).\]
\end{definition}


\noindent The quantities $\beta(L)$ and $\beta(M_L)$ are equal by the following lemma.

\begin{lemma}\label{lemma:X_L-M_L}
  The inclusion $X_L \to M_L$ induces an isomorphism $H_1(X_L;\Q(t)) \toiso H_1(M_L;\Q(t))$.
\end{lemma}

\begin{proof}
Define $\mu_L := \mu_{L_1} \sqcup \cdots \sqcup \mu_{L_m} \subset \partial X_L$, by taking $\mu_{L_i}$ to be an oriented meridian of the $i$th component $L_i$ of $L$.
Then $\partial X_L \cong \mu_L \times S^1$, and one forms the zero surgery $M_L$ by glueing $M_L = X_L \cup_{\mu_L \times S^1} \mu_L \times D^2$.

For each $i=1,\dots,m$, the representation of $\pi_1(M_L)$ restricted to $\pi_1(\mu_{L_i}) \to \Z$ is nontrivial (in fact it is an isomorphism).  Thus $H_*(\mu_{L_i};\Q(t)) =0$.  For the zeroth homology, this uses that $t-1$ is invertible in $\Q(t)$.  The K\"{u}nneth theorem then implies that $H_*(\mu_{L_i} \times Y;\Q(t)) =0$ for any finite CW-complex $Y$.  Thus for the disjoint union $\mu_L \times Y$ we also have $H_*(\mu_L \times Y;\Q(t)) =0$.  The Mayer-Vietoris sequence for $M_L = X_L \cup_{\mu_L \times S^1} \mu_L \times D^2$ then yields an isomorphism induced by the inclusion, $H_1(X_L;\Q(t)) \toiso H_1(M_L;\Q(t))$, as claimed.

\end{proof}



In the proof below of Theorem~\ref{theorem:LT-sig-4-ball-genus}, in light of Lemma~\ref{lemma:X_L-M_L}, we will use $\beta(L)$ in place of $\beta(M_L)$.  We need one more lemma before we begin the proof of the theorem.

\begin{lemma}\label{lemma:W_S-hom-Qt-coeffs}
  We have that $H_i(W_{\S};\Q(t))=0$ for $i=0,3,4$.
\end{lemma}
 \begin{proof}
   For $i=0$, this follows from~\cite[Proposition~2.9]{Cochran-Orr-Teichner:1999-1}, since the representation $\pi_1(W_\S) \to \Z$ is nontrivial.  For $i=3,4$, apply Poincar\'{e} duality and universal coefficients to see that $H_i(W_\S;\Q(t)) \cong H_{4-i}(W_\S,M_L;\Q(t))$, and then note that $H_{j}(W_\S,M_L;\Q) =0$ for $j=0,1$ and apply~\cite[Proposition~2.11]{Cochran-Orr-Teichner:1999-1}, which implies that also $H_{j}(W_\S,M_L;\Q(t))=0$ for $j=0,1$.
\end{proof}

%
%
%
%

\noindent Now we are ready to begin the proof.

\begin{proof}[Proof of Theorem~\ref{theorem:LT-sig-4-ball-genus}]
  We have:
\begin{equation} \label{eqn:a}
  2g-m+1 = \chi^\Q(W_\S) = \chi^{\Q(t)}(W_\S) =  \dim H_2(W_\S;\Q(t)) - \dim H_1(W_\S;\Q(t)).
\end{equation}
The first and last equalities follow from Lemma~\ref{lemma:homology-W-S} and Lemma~\ref{lemma:W_S-hom-Qt-coeffs} respectively.
As noted above, $H_1(W_{\S},M_L;\Q(t))=0$, which implies that the map $H_1(M_L;\Q(t)) \to H_1(W_{\S};\Q(t))$ is surjective.  Thus
\begin{equation}\label{eqn:c}
\beta(L) = \dim H_1(M_L;\Q(t)) \geq \dim H_1(W_\S;\Q(t)).
\end{equation}
Next,
\begin{align*}
 & \dim\big(\im \big(H_2(M_L;\Q(t)) \to H_2(W_\S;\Q(t))\big) \big) \\
= & \dim H_2(M_L;\Q(t)) - \dim \big( \im\big(H_3(W_\S,M_L;\Q(t)) \to H_2(M_L;\Q(t))\big) \big)\\
= & \dim H_2(M_L;\Q(t)) - \dim H_3(W_\S,M_L;\Q(t))\\
= & \beta(L) - \dim H_1(W_\S;\Q(t)).
\end{align*}
The first equality follows from the long exact sequence of the pair $(W_\S,M_L)$. The second equality also follows from this sequence, and the fact that $H_3(W_\S;\Q(t))=0$, so that the map $H_3(W_\S,M_L;\Q(t)) \to H_2(M_L;\Q(t))$ is injective.  The third equality follows from Poincar\'{e} duality, universal coefficients, and the fact that $\beta(L) = \beta(M_L) = \dim H_1(M_L;\Q(t))$.
We therefore obtain
\begin{align*}
  &\dim H_2(W;\Q(t))^{\dag} \\
  =& \dim H_2(W;\Q(t)) - \dim(\im H_2(M_L;\Q(t))) \\
  =& \dim H_2(W;\Q(t)) - \beta(L) + \dim H_1(W_\S;\Q(t))\\
 =& \dim H_2(W;\Q(t)) -\dim H_1(W_\S;\Q(t)) - \beta(L) + 2\dim H_1(W_\S;\Q(t))\\
 =& 2g-m+1 - \beta(L) + 2\dim H_1(W_\S;\Q(t))\\
 \leq & 2g-m+1 -\beta(L) + 2\beta(L)\\
  =& 2g-m+1+\beta(L).
\end{align*}
The first four equalities follow by definition, the computation above, algebra and equation~(\ref{eqn:a}) respectively.  The inequality follows from equation~(\ref{eqn:c}).
Finally, for any $z \in S^1$,
\begin{equation*}
|\sigma_L(z)|= |\sigma_{M_L}(z)| = |\sigma(\lambda_{W_{\S}}(z)) - \sigma(W_{\S})| = |\sigma(\lambda_{W_{\S}}(z))| \leq \dim H_2(W;\Q(t))^{\dag}.
\end{equation*}
The first equality is by Lemma~\ref{lemma:COT2}, the second by definition of $\sigma_{M_L}(z)$, and the final equality follows because $\sigma(W_\S)=0$ by Lemma~\ref{lemma:homology-W-S}.  The inequality follows from the fact that the absolute value of the signature of a form is always at most the dimension of the vector space on which it is defined.  Combining the above two displayed inequalities, we obtain
\[|\sigma_L(z)| \leq 2g-m+1+\beta(L).\]
Since this holds for any collection of surfaces $\S$ with boundary $L$, we can replace $g=g_4(L)$, and rearrange to arrive at the bound
\[|\sigma_L(z)|+m-1-\beta(L) \leq 2g_4(L)\]
as desired.  This completes the proof of Theorem~\ref{theorem:LT-sig-4-ball-genus}.
\end{proof}


\section{Characterisation of links with given four genera}\label{section:characterisation}

In this section we make preparations for the proof of Theorem~\ref{Thm:our_main} by giving a homological characterisation of the exterior in $D^4$ of a collection of disjointly embedded surfaces with a given set of genera.  \emph{From now on all homology groups will be with $\Z$ coefficients}, so we will omit the coefficients from the notation.

Let $\S = \S_1 \sqcup \dots \sqcup \S_m$ be a collection of oriented surfaces of genera $g_1,\dots,g_m$, each with a single $S^1$ boundary component.  
We write $g := \sum_{i=1}^m\, g_i.$  Let $X_L := S^3\setminus \nu L$ be the exterior of an oriented, ordered link $L = L_1 \sqcup \dots \sqcup L_m$ with all pairwise linking numbers vanishing i.e.\ $\lk(L_i,L_j)=0$ for $i \neq j$.  Recall  that a link bounds disjointly embedded surfaces in $D^4$ if and only if it has vanishing pairwise linking numbers.  Define
\[M_L^\S := X_L \cup_{\bigsqcup_m \, S^1 \times S^1} \S \times S^1,\]
with a meridian of $L_j$ identified with $\{\ast\} \times S^1$, for $\ast \in \partial \S_j$, and with a zero-framed longitude of $L_j$ identified with $\partial \S_j \times \{1\}$.

\begin{lemma}\label{Lemma:homology_M_K^g}
The homology of $M_L^\S$ is given by:
\[H_i(M_L^\S) \cong \begin{cases} \Z & i = 0,3 \\  \Z^{2g + m} & i=1,2 \\ 0 & \text{otherwise.} \end{cases}\]
\end{lemma}

\begin{proof}
A Mayer-Vietoris sequence computing $H_*(M_L^\S)$ is
\begin{align*}  \bigoplus_{j=1}^m \, H_{i}(S^1 \times S^1) \to H_i(X_L) \oplus \, \bigoplus_{j=1}^m \,H_i(\S_j \times S^1) \to H_i(M_L^\S) \to  \bigoplus_{j=1}^m\, H_{i-1}(S^1 \times S^1).\end{align*}
When $i=1$, we have
\[ \bigoplus_{j=1}^m \, \Z \oplus \Z \xrightarrow{\beta} \bigoplus_{j=1}^m \, (\Z \oplus (\Z^{2g_j} \oplus \Z)) \to H_1(M_L^\S) \xrightarrow{0} \bigoplus_{j=1}^m \,\Z,\]
since $H_1(\S_j \times S^1) \cong H_1(\S_j) \oplus H_1(S^1) \cong \Z^{2g_j} \oplus \Z$.  The map $\beta$ maps the $j$th summand to the $j$th summand, since the linking numbers of $L$ vanish.  On each summand, $\beta$ is given as follows.  The meridian of each torus maps to $(1,(0,1))$, while the longitude maps to zero, since it is a commutator of generators of $\pi_1(\S_j)$.  Therefore $H_1(M_L^\S) \cong \Z^{2g+m}$.
When $i=3$, we have:
\[0 \to H_3(M_L^\S) \to \bigoplus_{j=1}^m  H_2(S^1 \times S^1) \cong \Z^m \to H_2(X_L) \oplus\, \bigoplus_{j=1}^m  H_2(\S_j \times S^1). \]
We have that $H_2(X_L) \cong \Z^{m-1}$, and that $H_2(\S_j \times S^1) \cong H_1(\S_j) \otimes H_1(S^1) \cong H_1(\S_j)$.  The map $$\bigoplus_{j=1}^m \, H_2(S^1 \times S^1) \to \bigoplus_{j=1}^m \, H_2(\S_j \times S^1)$$ is the zero map, since the longitude of $L_j$ defines a trivial element of $H_1(\S_j)$.  The kernel of the map $$\bigoplus_{j=1}^m \, H_2(S^1 \times S^1) \to H_2(X_L)$$ is cyclic, generated by $(1,\dots,1) \in \bigoplus_{j=1}^m \, H_2(S^1 \times S^1)$.  The map $\bigoplus_{j=1}^m \, H_2(S^1 \times S^1) \to H_2(X_L)$ is onto.  Therefore $H_3(M_L^\S) \cong \Z$. The above descriptions of the maps in the Mayer-Vietoris sequence also imply that $H_2(M_L^\S)$ fits into the exact sequence
\[0 \to H_2(\S \times S^1) \cong \Z^{2g} \to H_2(M_L^\S) \to \Z^m \to 0,\]
where the final $\Z^m$ is the subgroup of $\bigoplus_{j=1}^m\, H_1(S^1 \times S^1)$ generated by the longitudes of the components of $L$.  Thus $H_2(M_L^\S) \cong \Z^{2g+m}$ as claimed.
\end{proof}

Next we give the following characterisation of links with four genera $g_1,\dots,g_m$.


\begin{proposition}\label{Prop:A}
An oriented, ordered link $L = L_1 \sqcup \dots \sqcup L_m$ bounds a collection of disjointly embedded surfaces in $D^4$ with genera equal to $(g_1,\dots,g_m)$ if and only if $M_L^\S$ is the boundary of a topological $4$-manifold $W$ that satisfies the following conditions.
\begin{enumerate}[(i)]
         \item\label{item:prop-A-i} On the subgroup $H_1(X_L) \leq H_1(M_L^\S)$, the inclusion induced map $j_* \colon H_1(M_L^\S) \to H_1(W)$ restricts to an isomorphism $$j_*| \colon H_1(X_L) \toiso H_1(W).$$
         \item\label{item:prop-A-ii} On the subgroup $H_2(\S \times S^1) \leq H_2(M_L^\S)$, the inclusion induced map $j_* \colon H_2(M_L^\S) \to H_2(W)$ restricts to an isomorphism $$j_*| \colon H_2(\S \times S^1) \toiso H_2(W) \cong \Z^{2g}.$$
         \item\label{item:prop-A-iii} The fundamental group $\pi_1(W)$  is normally generated by the meridians of $L$.
\end{enumerate}
\end{proposition}

\begin{proof}
First, we suppose that $L$ is the boundary of a collection of surfaces $\S = \S_1,\dots,\S_m$ of genera $g_1,\dots,g_m$, embedded in $D^4$, and define $W:= D^4 \setminus \nu \S$.  We claim that $W$ satisfies properties (\ref{item:prop-A-i}), (\ref{item:prop-A-ii}) and (\ref{item:prop-A-iii}).  A Mayer-Vietoris sequence computing $H_*(W)$ is:
\[H_{i+1}(D^4) \to H_i(\S \times S^1) \to H_i(W) \oplus H_i(\S \times D^2) \to H_i(D^4).\]
When $i=1$, this yields:
\[0 \to \bigoplus_j\, H_1(\S_j \times S^1) \toiso H_1(W) \oplus \bigoplus_j\, H_1(\S_j) \to 0.\]
Since $H_1(\S_j \times S^1) \cong H_1(\S_j) \oplus H_1(S^1) \cong \Z^{2g_j} \oplus \Z$, and the $H_1(\S_j)$ components map isomorphically to one another, we see that $H_1(W) \cong \Z^m$, generated by the $H_1(S^1)$ summands, which correspond to the meridians of $L$.  Since these meridians are also the generators of $H_1(X_L)$, this shows that $W$ satisfies (\ref{item:prop-A-i}).
When $i=2$, we have
\[0 \to \bigoplus_j\, H_2(\S_j \times S^1) \toiso H_2(W) \to 0.\]
Since, for each $j$, $H_2(\S_j) =0 = H_2(S^1)$, we have that $H_2(\S_j \times S^1) \cong H_1(\S_j) \otimes H_1(S^1) \cong \Z^{2g_j} \otimes \Z \cong \Z^{2g_j}$.  Therefore $H_2(W) \cong \Z^{2g}$.  We also note that $H_3(W) \cong H_3(\S \times S^1) \cong 0$.  This shows that $W$ satisfies~(\ref{item:prop-A-ii}).

Now we use the Seifert Van-Kampen theorem to investigate the fundamental group of $W$.  Define $W_j := D^4 \setminus (\bigcup_{i=1}^j \, \S_j)$, so that $W = W_m$. We claim that, for all $j$, and so in particular for $j=m$, we have:
\[\frac{\pi_1(W_j)}{\langle \langle \mu_1,\dots\mu_j  \rangle\rangle} \cong \{1\},\]
where $\mu_i$ is a fundamental group element given by a meridian of the $i$th component of $L$.  We proceed by induction.  For the base case:
\begin{align*} \{1\} &\cong \pi_1(D^4) \cong \pi_1(\S_1 \times D^2) \ast_{\pi_1(\S_1 \times S^1)} \pi_1(W_1) \\
& \cong \pi_1(\S_1) \ast_{\pi_1(\S_1) \times \pi_1(S^1)} \pi_1(W_1) \\
& \cong \{1\} \ast_{\pi_1(S^1)} \pi_1(W_1)
 \cong \frac{\pi_1(W_1)}{\langle\langle \mu_1 \rangle\rangle},
\end{align*}
The penultimate isomorphism in the sequence follows from the observation that $\pi_1(\S_1) \times \pi_1(S^1) \to \pi_1(\S_1)$ is surjective with kernel $\pi_1(S^1)$.
For the inductive step, we show that
\[\frac{\pi_1(W_{j-1})}{\ll\ll \mu_1,\dots,\mu_{j-1}\rr\rr} \cong \frac{\pi_1(W_{j})}{\ll\ll \mu_1,\dots,\mu_{j}\rr\rr}. \]
To see this, follow a similar calculation to that above, to yield:
\[\pi_1(W_{j-1}) \cong \pi_1(\S_{j} \times D^2) \ast_{\pi_1(\S_{j} \times S^1)} \pi_1(W_{j}) \cong \frac{\pi_1(W_j)}{\langle\langle \mu_j \rangle\rangle}.\]
Take the quotient of both sides by $$\ll\ll \mu_1,\dots,\mu_{j-1}\rr\rr \cong \ll\ll \mu_1,\dots,\mu_j\rr\rr/\ll\ll \mu_j \rr\rr$$ to yield the iterative step and therefore the claim.
The above implies that $\pi_1(W)$ is normally generated by the meridians of $L$, which shows that $W$ satisfies (\ref{item:prop-A-iii}).  This completes the proof of the only if part of the proposition.

Now, suppose that $M_L^\S$ bounds a topological four manifold $W$ as in the statement of the proposition.  We shall prove that $L$ bounds a collection of locally flat oriented embedded surfaces with genera~$(g_1,\dots,g_m)$.  To begin, define another 4-manifold $$D:= W \cup_{\S \times S^1 \subseteq M_L^\S} \S \times D^2.$$
Any self-homeomorphism of each $\S_j$ that is the identity on the boundary will suffice for this construction, since a homeomorphism induces isomorphisms on homology and on fundamental groups.
Note that
\[\partial D = M_L^\S \sm \S \times S^1 \cup_{\partial \S \times S^1} \partial \S \times D^2 = S^3.\]
We calculate the homology and the fundamental group of $D$.  For the homology, we have the Mayer-Vietoris sequence
\[H_{i+1}(D) \to \bigoplus_{j=1}^m \, H_i(\S_j \times S^1) \toiso H_i(W) \oplus \bigoplus_{j=1}^m \, H_i(\S_j \times D^2) \to H_i(D) \to\]
for $i\geq 1$, with $H_1(D) \xrightarrow{0} H_0(\S \times S^1)$.  The central map is an isomorphism for $i=1$ since $H_1(\S_j \times S^1) \cong H_1(\S_j) \oplus H_1(S^1)$; the $H_1(\S_j)$ summands map isomorphically to the $H_1(\S_j \times D^2)$ summands, while the $H_1(S^1)$ summands collectively map isomorphically to $H_1(W)$ and by the zero map to $H_1(\S_j \times D^2)$, by property (\ref{item:prop-A-i}).  For $i=2$, $H_2(\S \times S^1)$ maps isomorphically to $H_2(W)$, by (\ref{item:prop-A-ii}).  Poincar\'{e}-Lefschetz duality and the universal coefficient theorem imply that $H_3(W) \cong H^1(W,M_L^\S) \cong \Hom_{\Z}(H_1(W,M_L^\S),\Z) \cong 0$.  Therefore $\wt{H}_*(D) = 0$.  For the fundamental group, we define  $W_j := W \cup (\bigcup_{i=0}^j\, \S_j \times D^2)$, with $\S_0 := \emptyset$, so that $W_0 = W$ and $W_m = D$. Again using the Seifert Van-Kampen theorem we have
\begin{align*}
\pi_1(W_{j}) & \cong \pi_1(W_{j-1}) \ast_{\pi_1(\S_j \times S^1)} \pi_1(\S_j \times D^2)\\
& \cong \pi_1(W_{j-1}) \ast_{\pi_1(\S_j) \times \pi_1(S^1)} \pi_1(\S_j)  \cong \frac{\pi_1(W_{j-1})}{\ll\ll \pi_1(S^1) \rr\rr}
\end{align*}
By (\ref{item:prop-A-iii}) and induction, we therefore have that $\pi_1(D) \cong \{1\}$.  Since $D$ has the homotopy groups of a 4-ball and $\partial D =S^3$, by Freedman's 4-dimensional topological $h$-cobordism theorem we deduce that in fact $D$ is homeomorphic to $D^4$~\cite{Freedman-Quinn:1990-1}.  The image of $\S \times \{0\}$ under this homeomorphism produces the required embedded surfaces.  This completes the proof of Proposition \ref{Prop:A}.
\end{proof}

\section{Proof of the infection by a string link theorem}\label{section:proof-of-infection-by-string-link-theorem}

In this section we give the proof of Theorem~\ref{Thm:our_main}.  As readers of \cite{CFT09} will recognise, the proof proceeds by constructing a 4-manifold $N'$ for the infection link that satisfies the conditions of Proposition~\ref{Prop:A}.  First we construct a 4-manifold $N$, whose boundary and whose second homology is slightly too big, and then we improve it to $N'$ by capping off the extra boundary with a special topological 4-manifold that is homotopy equivalent to a wedge of circles.  This will require the remainder of the article.  In this section, as in the previous section, all homology and cohomology groups are with $\Z$ coefficients.

\subsection{The $4$-manifold $N$}

Let $L= L_1 \sqcup \dots \sqcup L_m$ be an oriented, ordered link in $S^3$ with exterior $X_L := S^3\setminus \nu L$, which bounds locally flat, oriented, disjointly embedded surfaces $\S = \S_1 \sqcup \dots \sqcup \S_m$ in $D^4$.  Define $Y_{\S} := D^4 \setminus \nu \S,$
and let $\phi \colon \mathbb{E} \to S^3$ be an embedding of an $r$-multi-disc $\mathbb{E}$, whose image we denote by $\mathbb{E}_\phi$, such that the axes $\eta_1,\dots,\eta_r$  are closed curves in $X_L$ with $[\eta_k] = [1] \in \pi_1(Y_{\S})$ for all $k$.  Since $\eta_k$ is null-homotopic, putting a null homotopy in general position yields an immersed disc in $Y_{\S}$.  We also arrange these discs to be in general position with respect to each other.  Let $c$ be the total number of intersection and self-intersection points amongst these discs.  Let $J$ be a string link whose closure $\widehat{J}$ has vanishing Milnor's $\ol{\mu}$-invariants of length up to and including $2c$, that is $\ol{\mu}_{\widehat{J}}(I) =0$ for $|I| \leq 2c$.

Denote the image under $\phi$ of the complement of the $r$ sub-discs of the multi-disc $\mathbb{E}$ by $E_\phi$.  The space $E_\phi \times I \subseteq \partial Y_{\S}$ is a handlebody with $r$ 1-handles.  Let $M_{J}$ be the zero surgery on the closure $\widehat{J}$ of the string link $J$.  This zero surgery decomposes into the union of the exterior of $J$, the exterior of a trivial string link, and the $r$~solid tori from the zero surgery.  The exterior of a trivial string link with $r$ components is also a handlebody with $r$ 1-handles.  Denote its image in $M_J$ by $\mathbb{H} \subseteq M_J$.  Then identify these two handlebodies to form the union
\[N := Y_{\S} \cup_{E_\phi \times I = \mathbb{H} \subseteq M_J \times \{0\}} M_J \times [0,1]. \]
In what follows, let $S := S(L,\mathbb{E}_\phi,J)$ be the infection of $L$ by the string link $J$, with $r$-multi-disc $\mathbb{E}_{\phi}$ and axes $\eta_1,\dots,\eta_r$. For more details on this construction see \cite[Section~2.2]{CFT09}; the above is a summary of their exposition, with similar notation.  The main difference is that our $Y_{\S}$, which corresponds to their $W_L$, is the exterior of a collection of surfaces rather than the exterior of a collection of slice discs.

\begin{proposition}\label{Propn:Properties_of_N}
The $4$-manifold $N$ is such that:
\begin{enumerate}[(1)]
\item\label{item:properties-of-N-1} $\partial N = M_S^\S \sqcup -M_J$;
\item\label{item:properties-of-N-2} $\pi_1(N)$ is normally generated by the meridians of $S$;
\item\label{item:properties-of-N-3} $\im(\pi_1(M_J) \to \pi_1(N)) = \{1\}$;
\item\label{item:properties-of-N-4} the composition $H_1(X_S) \to H_1(M_S^\S) \to H_1(N)$ is an isomorphism; and
\item\label{item:properties-of-N-5} $H_2(M_J) \rightarrowtail H_2(N)$ is injective, with $H_2(N) \cong \Z^{2g+m}.$
\end{enumerate}
\end{proposition}

\begin{proof}
Property (\ref{item:properties-of-N-1}) follows directly from the construction of $N$.  The Seifert-Van Kampen theorem gives us that
\[\pi_1(N) \cong \pi_1(Y_{\S}) \ast_{\pi_1(\mathbb{H})} \pi_1(M_J \times [0,1]),\]
where $\pi_1(\mathbb{H}) \cong \S_r$, the free group on $r$ letters, which is the group along which we amalgamate, is generated by the meridians of $J$.  Note that $\pi_1(M_J)$ is normally generated by the meridians of $J$.  Then recall that the meridian of the $k$th component of $J$ is identified with the curve $\eta_k$, and that our hypothesis is that each $\eta_k$ is null-homotopic in $Y_{\S}$.  Therefore $\pi_1(N) \cong \pi_1(Y_{\S})$.  Then, since a meridian of $L$ becomes a meridian of $S$ during the infection construction, we have proved~(\ref{item:properties-of-N-2}) and~(\ref{item:properties-of-N-3}).

Next, we calculate the homology of $N$.  The Mayer--Vietoris sequence yields:
\[H_1(\mathbb{H}) \to H_1(Y_{\S}) \oplus H_1(M_J) \to H_1(N) \to 0,\]
which translates to
\[\Z^r \xrightarrow{\left(\ba{c} 0 \\ \Id \ea\right)} \Z^m \oplus \Z^r \to H_1(N) \to 0.\]
There the first component of the first map is zero because the axis curves $\eta_k$ are null homotopic in $Y_{\S}$.
Therefore $H_1(N) \cong H_1(Y_{\S}) \cong \Z^m$.  Since the meridians of $L$ generate the homology $H_1(Y_{\S})$, it follows that the meridians of $S$ generate the homology $H_1(N)$.  This proves (\ref{item:properties-of-N-4}).  Another portion of the same Mayer--Vietoris sequence is the following:
\[H_2(\mathbb{H}) \to H_2(Y_{\S}) \oplus H_2(M_J) \to H_2(N) \xrightarrow{0} H_1(\mathbb{H}).\]
Since $H_2(\mathbb{H}) \cong 0$, this implies that
\[H_2(N) \cong H_2(Y_{\S}) \oplus H_2(M_J) \cong \Z^{2g} \oplus \Z^m \cong \Z^{2g+m},\]
Note that since the $\Z^{2g}$ summand of $H_2(N)$ comes from $H_2(Y_{\S})$, it is the image of the inclusion of $H_2(\S \times S^1)$, by the only if part of Proposition \ref{Prop:A} (\ref{item:prop-A-ii}). This completes the proof of (\ref{item:properties-of-N-5}) and therefore of Proposition~\ref{Propn:Properties_of_N}.
\end{proof}

We need to cap off the boundary component $M_J \times \{1\}$ of $N$, and we need to do so in such a way that $H_2(M_J)$ is killed, in order to construct a 4-manifold satisfying the conditions of Proposition~\ref{Prop:A} with respect to $M_S^\S$.  The next subsection outlines the construction which improves $N$ to a new four manifold $N'$. The subsection after that proves that $N'$ satisfies the conditions of Proposition \ref{Prop:A}.

\subsection{Improving $N$ to $N'$}

Since $\eta_1,\dots,\eta_r$ are null homotopic in $Y_{\S}$, they each bound an immersed disk in $Y_{\S}$, as observed above.  Recall that $\eta_k$ is identified with a meridian $\mu_k$ of the $k$th component of $J$ in $\mathbb{H} \subseteq M_J$.  Denote the immersed discs bounded by the $\eta_k$, together with, for each $k$, collars $\mu_k \times I \subset M_J \times I$, by $\delta_1,\dots,\delta_r$.  Take a regular neighbourhood of each disc in $N$, $\nu\delta_k$, and take its union with a collar of $M_J$,
\[M_J \times (1-\eps,1] \subseteq M_J \times [0,1].\]
Denote this union by
\[M_1 := \big(M_J \times (1-\eps,1]\big) \cup \bigcup_{k=1}^r\, \nu\delta_k. \]
Consider $N\setminus M_1$, and partition the boundary of $\cl M_1$ as $\partial (\cl M_1) = \partial^+ M_1 \sqcup \partial^- M_1$, where $\partial^- M_1 := M_J$.  Here we take the closure of $M_1$ inside $N$.  Therefore,
\[\partial(N \setminus M_1) \cong M_S^\S \sqcup \partial^+ M_1.\]
The following is \cite[Lemma~2.7]{CFT09}, although we remark that while (\ref{item:M3-c}) is only shown in the proof of \cite[Lemma~2.7]{CFT09}, here we have promoted it to a property.

\begin{lemma}\label{lemma:CFTtechnical_lemma}
There exists a 4-manifold $M_3$ with $\partial M_3 \cong \partial^+ M_1$, such that:
\begin{enumerate}[(a)]
\item\label{item:M3-a} The inclusion of the boundary induces an isomorphism $$H_1(\partial M_3) \toiso H_1(M_3).$$
\item\label{item:M3-b} The 4-manifold $M_3$ is homotopy equivalent to a wedge of $c$ circles, one for each double point of $\bigcup_{k=1}^r\,\delta_k$, with the fundamental group generated by the double point loops.
\item\label{item:M3-c} The meridians of the $\delta_k$ are null-homotopic in $M_3$.
\end{enumerate}
\end{lemma}

This is a main technical lemma of \cite{CFT09}. It relies on techniques and results of \cite{FT95II}, which makes crucial use of \cite{Freedman-Quinn:1990-1}, in particular Section~5.3 and Chapter~6.
The proof of Lemma~\ref{lemma:CFTtechnical_lemma} starts with a candidate manifold with nonzero $\pi_2$, that one wishes to excise.  Since the fundamental group is free, it is not known to satisfy the $\pi_1$-null disc lemma, so surgery on embedded $2$-spheres to kill $\pi_2$ is not possible directly.  However, it is possible to find an $s$-cobordant manifold in which surgery is possible, provided the spheres in question are $\pi_1$-null.  Here $\pi_1$-null means that the fundamental group of the image of the spheres maps trivially into the fundamental group of the ambient manifold.  The assumption on Milnor's invariants is used in \cite{CFT09} to achieve $\pi_1$-nullity.  Since all we need is some $4$-manifold with the right homotopy type and the right boundary, this suffices.

With $M_3$ as in Lemma~\ref{lemma:CFTtechnical_lemma}, we define:
\[N' := N \setminus M_1 \cup_{\partial^+ M_1} M_3.\]
In the next subsection we show that $N'$ satisfies the conditions of Proposition \ref{Prop:A}.

\subsection{The $4$-manifold $N'$ satisfies Proposition \ref{Prop:A}}

As indicated by the title, this subsection contains the proof of the following proposition.

\begin{proposition}\label{prop:N-satisties-properties}
The $4$-manifold $N' := N \setminus M_1 \cup_{\partial^+ M_1} M_3$ satisfies the conditions of Proposition \ref{Prop:A} with respect to the link~$S$.
\end{proposition}

\begin{proof}
First, we seek to understand the homology of $\partial^+ M_1$, $N\setminus M_1$, and we aim to understand the inclusion induced maps $H_*(\partial^+ M_1) \to H_*(N \setminus M_1)$ and $\pi_1(\partial^+ M_1) \to \pi_1(N \setminus M_1)$.  Following \cite{CFT09}, note that $N \setminus M_1 \cong N\setminus \cup_{k=1}^r\,\nu \delta_k$, where $\nu \delta_k$ is a regular neighbourhood of $\delta_k$.  The boundary of $\cup_{k=1}^r\,\nu\delta_k$ splits as
\[\partial\big(\cup_{k=1}^r\,\nu\delta_k\big) = \cup_{k=1}^r\,\nu(\partial\delta_k) \cup \partial'\big(\cup_{k=1}^r\,\nu\delta_k\big),\]
where, for each $k$, $\nu(\partial\delta_k) \cong S^1 \times D^2$ and $\partial'(\cup_{k=1}^r\,\nu\delta_k)$ is what remains of the boundary.  Then, by excision and Poincar\'{e}-Lefschetz duality, we have that:
\begin{align*}
H_i(N,N\setminus M_1) &\cong H_i(N,N\setminus \cup_{k=1}^r\,\nu \delta_k) \cong H_i(\cup_{k=1}^r\,\nu\delta_k,\partial'(\cup_{k=1}^r\,\nu\delta_k)) \\ &\cong H^{4-i}(\cup_{k=1}^r\,\nu\delta_k,\cup_{k=1}^r\,\nu(\partial\delta_k)) \cong \begin{cases}\Z^r & i=2 \\ \Z^c & i=3 \\ 0 & \text{otherwise}. \end{cases}
\end{align*}

For $i=2$, the homology $H_2(N,N\setminus M_1)$ is generated by the images, for each $k$, of maps $\{\ast\} \times D^2 \subseteq D^2 \times D^2 \to \nu \delta_k$.  These are transverse discs to the $\delta_k$.

In a neighbourhood $U \cong D^4$ of a double point of $\bigcup_{k=1}^r\,\delta_k$, recall that the two sheets intersect $\partial U \cong S^3$, which is the boundary of a neighbourhood of the double point, in a Hopf link $K_0 \cup K_1$.  The exterior $S^3 \setminus \nu(K_0 \sqcup K_1)$ of this Hopf link is homeomorphic to $S^1 \times S^1 \times I$, and $S^1 \times S^1 \times \{1/2\}$ is by definition a Clifford torus of the double point.  The solid torus $S^1 \times S^1 \times [0,1/2] \cup \nu K_0$ contains one component of the Hopf link $K_0$ as its core.  In the case of a self-intersection of a disc $\delta_k$, the double point loop is a loop on $\delta_k$ which starts and ends at the double point, leaving and returning on different sheets and avoiding all other double points.   In the case of an intersection between discs $\delta_k$ and $\delta_{k'}$, the double point loop leaves the intersection point along $\delta_k$ and returns along $\delta_{k'}$, joining up between $\eta_k$ and $\eta_{k'}$ in $M_J$.  A double point loop intersects $K_0$ at a single point.  The solid tori $S^1 \times S^1 \times [0,1/2] \cup \nu K_0$ described above, one for each double point of $\bigcup_{k=1}^r\,\delta_k$, generate $H_3(N,N\setminus M_1)$, since they can be kept within $\bigcup_{k=1}^r\,\nu \delta_k$.

The long exact sequence of the pair $(N,N \sm M_1)$ therefore yields:
\[\to \Z^c \xrightarrow{\partial} H_2(N \setminus M_1) \to H_2(N) \to \Z^r \xrightarrow{\partial} H_1(N \setminus M_1) \to H_1(N) \to 0.\]
Recall that $H_2(N) \cong \Z^{2g+m}$ and $H_1(N) \cong \Z^m$.  A generator of $H_2(N)$ which is the image of a generator of $H_2(M_J)$ in Proposition \ref{Propn:Properties_of_N} (5), is a capped off Seifert surface for a component of $J$.  The capping is done by a disc which becomes a transverse disc to $\delta_k$ in $\nu(\partial \delta_k)$, i.e.\ $\{\ast\} \times D^2 \subseteq D^2 \times D^2$ with $\ast \in S^1$.  Therefore the map $H_2(N) \twoheadrightarrow \Z^r= H_2(N,N\sm M_1)$ is surjective.  This implies that
\[H_1(N \setminus M_1) \cong H_1(N) \cong \Z^m. \]
Note that the other generators of $H_2(N)$, those which are the image of generators of $H_2(Y_{\S}) \cong \Z^{2g}$, map trivially into $H_2(N, N \setminus M_1) \cong \Z^r$.  This gives us an exact sequence
\[\Z^c \to H_2(N \setminus M_1) \to \Z^{2g} \to 0.\]
Recall from Proposition~\ref{Propn:Properties_of_N}~(\ref{item:properties-of-N-2}) that $\pi_1(N)$ is normally generated by the meridians of $S$.  We can assume that any homotopies are transverse to the $\delta_k$.  Therefore $\pi_1(N \setminus M_1)$ is normally generated by the meridians of $S$ and the meridians of the~$\delta_k$.

Now that we understand the homology and the fundamental group of $N \setminus M_1$, we are in a position to calculate the homology and the fundamental group of~$N'$.  Recall that~$M_3$ is a 4-manifold which is homotopy equivalent to $\vee_c S^1$, corresponding to the double point loops of the intersections and self-intersections of the $\delta_k$, with $\partial M_3 = \partial^+ M_1$, and $H_1(\partial M_3) \toiso H_1(M_3) \cong \Z^c$ an isomorphism.  The Mayer-Vietoris sequence for $N' = N \setminus M_1 \cup M_3$ yields the exact sequence:
\[H_1(\partial M_3) \to H_1(N \setminus M_1) \oplus H_1(M_3) \to H_1(N') \to 0.\]
Since $H_1(\partial M_3) \toiso H_1(M_3)$ is an isomorphism, we have that $H_1(N\setminus M_1) \toiso H_1(N')$.  Together with the isomorphism $H_1(N\setminus M_1) \toiso H_1(N)$ from above, this shows that property (\ref{item:prop-A-i}) of Proposition \ref{Prop:A} is satisfied.

Next, the Mayer-Vietoris sequence for $N' = N \setminus M_1 \cup M_3$ also gives rise to the exact sequence
\[H_2(\partial M_3) \to H_2(N \setminus M_1) \oplus H_2(M_3) \to H_2(N') \to 0.\]
Since $H_2(M_3) \cong 0$, we have that
\[H_2(N') \cong \coker(H_2(\partial M_3) \to H_2(N\setminus M_1)).\]
Recall that $H_2(N\setminus M_1)$ sits in the exact sequence
\[\Z^c \to H_2(N \setminus M_1) \to \Z^{2g} \to 0,\]
where the $\Z^c$ is generated by solid tori described above associated to each double point, and the map $\Z^c \to H_2(N \setminus M_1)$ is the boundary map from the long exact sequence of a pair.  The Clifford tori therefore generate the image of $\Z^c$ in $H_2(N \setminus M_1)$.  But these Clifford tori can be assumed to live in $\partial^+ M_1$, which is glued to $\partial M_3$.  Moreover, since $\partial M_3$ is a closed 3-manifold, Poincar\'{e} duality forces $H_2(\partial M_3) \cong \Z^c$. Therefore $$H_2(N') \cong \coker(H_2(\partial M_3) \to H_2(N\setminus M_1)) \cong H_2(N \setminus M_1)/\im(\Z^c) \cong \Z^{2g},$$ generated by the image of $H_2(Y_{\S})$ as required.
We have now shown that capping off with $M_3$ serves to kill the generators of second homology which came from $H_2(M_J)$.  As in the proof of Proposition \ref{Propn:Properties_of_N}, the remaining $\Z^{2g}$ summand is the isomorphic image of $H_2(\S \times S^1)$.  Therefore property (\ref{item:prop-A-ii}) of Proposition \ref{Prop:A} is satisfied.

Finally, recall that $\pi_1(M_3)$ is generated by double point loops, and note that these double point loops come from the boundary.  So  $$\pi_1(\partial M_3) \twoheadrightarrow \pi_1(M_3)$$ is surjective.  Therefore, since
\[\pi_1(N') \cong \pi_1(N\setminus M_1) \ast_{\pi_1(\partial M_3)} \pi_1(M_3),\]
we have that $\pi_1(N\setminus M_1)$ surjects onto $\pi_1(N')$.  Since $\pi_1(N\setminus M_1)$ is normally generated by the meridians of $S$ and of the $\delta_k$, so is $\pi_1(N')$.
Recall from Lemma~\ref{lemma:CFTtechnical_lemma}~(\ref{item:M3-c}) that the meridians of the $\delta_k$ are null homotopic in $M_3$.
%
%
Therefore $\pi_1(N')$ is normally generated by the meridians of $S$, so $N'$ satisfies property (\ref{item:prop-A-iii}) of Proposition \ref{Prop:A}.
This completes the proof that $N'$ is a 4-manifold with boundary $M_S^\S$, satisfying the conditions of Proposition~\ref{Prop:A}.
\end{proof}

Therefore by Propositions~\ref{prop:N-satisties-properties} and~\ref{Prop:A}, the link $S$ bounds a collection of disjointly embedded surfaces in $D^4$ with genera $g_1,\dots,g_m$, as claimed.  This completes the proof of Theorem \ref{Thm:our_main}.

\subsection*{Acknowledgements}

Sebastian Baader and Lukas Lewark pointed out that explicit treatment of the relationship between the locally flat 4-genus and Levine-Tristram signatures in the literature has perhaps been insufficient.  I am also indebted to Lukas Lewark for showing me the Seifert surface approach to finding a genus four surface for the example in Section~\ref{section:example}.
In addition, I would like to thank Jae Choon Cha, Peter Feller, Stefan Friedl, Pat Gilmer, Chuck Livingston, Matthias Nagel and the referee for their helpful comments and suggestions.  I am grateful to the Max Planck Institute for Mathematics in Bonn for its hospitality during the time when this paper was written.  I am supported by the NSERC grant ``Structure of knot and link concordance.''

\bibliographystyle{alpha}
\def\MR#1{}
\bibliography{research}

\end{document}